\documentclass{ps}

\usepackage[utf8]{inputenc}
\usepackage[T1]{fontenc}
\usepackage{textcomp}
\usepackage{dsfont}
\usepackage{hyperref}
\usepackage{color}
\usepackage{graphicx}
\usepackage{epstopdf}
\usepackage{amsmath,amsopn,amscd,amsthm,amssymb}
\usepackage{enumitem}
\usepackage{here}
\usepackage[dvipsnames]{xcolor}
\usepackage{comment}

\usepackage{euscript}
\usepackage[final]{pdfpages}
\usepackage[utf8]{inputenc}
\usepackage[T1]{fontenc}

\def \R{\mathbb{R}}
\def\C{\mathbb{C}}
\def\E{\mathbb{E}}
\def\P{\mathbb{P}}

\def\cL{\mathcal{L}}

 \usepackage{amsmath}
 \usepackage{amsfonts}

\newtheorem{Remark}{Remark}
\newtheorem{proposition}{Proposition}

\newtheorem{coro}{Corollary}
\newtheorem{conj}{Conjecture}
\newtheorem{thm}{Theorem}

\usepackage{graphicx}

\def \1{\mathds{1}}

\begin{document}
\title{A probabilistic point of view for the Kolmogorov hypoelliptic equations}
%
\author{P. Etor\'{e}}\address{Univ. Grenoble Alpes, CNRS, Inria, Grenoble INP, LJK, Grenoble, France
E-mail \texttt{pierre.etore@univ-grenoble-alpes.fr}}
\author{J. R. Le\'{o}n}\address{Univ. de La República. IMERL. Montevideo, Uruguay, Escuela de Matemática UCV. Venezuela
E-mail \texttt{rlramos@fing.edu.uy}}
\author{C. Prieur}\address{Univ. Grenoble Alpes, CNRS, Inria, Grenoble INP, LJK, Grenoble, France E-mail\texttt{ clementine.prieur@univ-grenoble-alpes.fr}}
%
%
\begin{abstract}
In this work, we propose a method for solving Kolmogorov hypoelliptic equations based on Fourier transform and  Feynman-Kac formula. We first explain how the Feynman-Kac formula can be used to compute the fundamental solution to parabolic equations with linear or quadratic potential. Then applying these results after a Fourier transform we deduce the computation of the solution to a first class of Kolmogorov hypoelliptic equations. Then we solve partial differential equations obtained via Feynman-Kac formula from the Ornstein-Uhlenbeck generator. Also, a new small time approximation of the solution to a certain class of Kolmogorov hypoelliptic equations is provided. We finally present the results of numerical experiments to check the practical efficiency of this approximation.\\
\begin{center}{\bf R\'esum\'e}\end{center}
Dans ce travail, nous proposons une méthode de résolution des équations hypoelliptiques de Kolmogorov basée sur la transformée de Fourier et la formule de Feynman-Kac. Nous expliquons d'abord comment la formule de Feynman-Kac peut être utilisée pour calculer la solution fondamentale des équations paraboliques à potentiel linéaire ou quadratique. Puis en appliquant ces résultats après une transformée de Fourier, nous déduisons le calcul de la solution d'une première classe d'équations hypoelliptiques de Kolmogorov. Ensuite, nous résolvons des équations aux dérivées partielles obtenues via la formule de Feynman-Kac à partir du générateur d'Ornstein-Uhlenbeck. De plus, une nouvelle approximation en temps petit de la solution d'une certaine classe  d' équations hypoelliptiques de Kolmogorov est \'etablie. Nous présentons enfin les résultats d'expériences numériques pour vérifier l'efficacité pratique de cette approximation.
\end{abstract}


\maketitle

\section*{Introduction}\label{introduction}

In the forties of the twentieth century, two important tools in Physics were developed. The first one was the path formulation of quantum mechanics by Feynman and the second, which appeared sometime later, is the adaptation by Kac of Feynman's ideas to heat equation.  Feynman introduced in his work a formal path integral, defined over trajectories, showing the equivalence between his approach of quantum mechanics and the one developed by Schrödinger.  Kac, substituting the formal Feynman's path integral by an integral over Brownian motion paths, achieved to solve the heat equation with a potential $V$ by means of the expectation of a Brownian motion functional. The solution found by Kac is called Feynman-Kac formula  (F-K) in recognition to the work of both researchers. Since then both techniques have been very useful in mathematical physics and a huge amount of literature has been published since its introduction. Two important references linked with these matters are the books \cite{fey:Hibs} and \cite{Kac:Kac} written by Feynman and Kac, respectively. For an overview on the topic, the reader can consult \cite{Sch:Sch} and the bibliography therein.

In the present work we will use the F-K formula to find the fundamental solution to several hypoelliptic Partial Differential Equations (PDE). Let us consider two vector fields $c:\R^d\to\R^{\tilde d}$ and $b:\R^d\to\R^d$, with $d$ not necessarily equal to $\tilde d$.
We consider the following PDE
\begin{equation}
\label{Kolm}
\left\{\begin{array}{rcl}
\frac{\partial u}{\partial t}(t,x,y) & = & \frac12\Delta_yu(t,x,y)+<b(y),\nabla_y u(t,x,y)>+<c(y),\nabla_xu(t,x,y)> \\
 & &  +  \alpha V(x,y)u(t,x,y),
\quad (t,x,y)\in\R_+^*\times\R^{\tilde d}\times\R^d\\
u(0,x,y) &  = &  f(x,y),\quad (x,y)\in\R^{\tilde d}\times\R^d.
\end{array}\right.
\end{equation}

By a fundamental solution to the PDE \eqref{Kolm} we mean a kernel
$p(t,x,y,x',y')$ such that for any initial condition $f$ satisfying mild conditions one has
\begin{equation}
    \label{eq:sol-fund}
    u(t,x,y)=\int_{\R^{\tilde d}\times\R^d}p(t,x,y,x',y')f(x',y')dx'dy'.
\end{equation}
Using differentiation under the integral sign, the fundamental solution can be seen, for any $(x',y')\in \R^{\tilde d}\times\R^d$, as the solution to
\begin{eqnarray}\label{Kolm-fund}
\partial_tp(t,x,y,x',y')&=&\frac12\Delta_yp(t,x,y,x',y')+<b(y),\nabla_y p(t,x,y,x',y')>\\ 
\nonumber &&\hspace{1cm}+<c(y),\nabla_xp(t,x,y,x',y')>+\alpha V(x,y)p(t,x,y,x',y'),
\end{eqnarray}
for any $(t,x,y)\in\R_+^*\times\R^{\tilde d}\times\R^d$ and with $p(0,x,y,x',y')= \delta_{x'}(x)\otimes\delta_{y'}(y)$, which means
that
$$\lim_{t\downarrow 0}\int p(t,x,y,x',y')f(x',y')dx'dy'=f(x,y).$$ 
Note that
the solution to \eqref{Kolm-fund} is sometimes taken as a definition of the fundamental solution in the literature (see, e.g., \cite{friedman-eds} in the elliptic case).

Let us now consider the following system of Stochastic Differential Equations (SDE)

\begin{equation}\label{hypo1}
\begin{cases}
dX(t)  =  c(Y(t))dt \\
dY(t)  =  dW(t)+b(Y(t))dt
\end{cases} 
\end{equation}
where $W$ is some $d$-dimensional Brownian motion.

It is well known that if \eqref{Kolm} and \eqref{hypo1} have both a unique solution, one has (under mild assumptions on~$f$ and $V$) the probabilistic representation
\begin{equation}
    \label{eq:rep-sto}
    u(t,x,y)=\E^{x,y}\big[ \,e^{\alpha \int_0^tV(X(s),Y(s))ds}f(X(t),Y(t)) \,  \big]
\end{equation}
(see, e.g., \cite[Section 5.7]{kara}; one can adapt these results to the hypoelliptic case). 
Here $\E^{x,y}$ denotes the expectation computed under $\P(\cdot\,|X(0)=x,Y(0)=y)$.
Formula \eqref{eq:rep-sto} is a generalization of the initial formula by Feynman and Kac.

In these notes, the procedure we propose consists in rewriting the expectation in \eqref{eq:rep-sto} by using probabilistic tricks such as the multidimensional-complex version of the Cameron-Martin-Girsanov formula \cite{Az:Doss} and/or a Gaussian regression argument. Then, by comparing \eqref{eq:sol-fund} and
\eqref{eq:rep-sto}, we will deduce the fundamental solution to various PDEs of interest of type \eqref{Kolm} (see first examples in Section \ref{sec:FK}).
One particular case of interest is the following: if $\alpha=0$ in \eqref{Kolm} then 
\begin{eqnarray}\label{eq:transition}
\nonumber u(t,x,y)&=&\int_{\R^{\tilde d}\times\R^d}p(t,x,y,x',y')f(x',y')dx'dy'\\
&=&\E\big[ \,f(X(t),Y(t)) \, |\, X(0)=x,Y(0)=y\big]
\end{eqnarray}
and the fundamental solution $p(t,x,y,x',y')$ clearly appears as the transition function of the process $(X,Y)$. Then, as for \eqref{Kolm-fund}, this transition function solves for any arrival point~$(x',y')\in \R^{\tilde d}\times\R^d$ the PDE
\begin{equation}\label{Kolm-back}
\left\{\begin{array}{rcl}
\partial_tp(t,x,y,x',y')&=&\frac12\Delta_yp(t,x,y,x',y')+<b(y),\nabla_y p(t,x,y,x',y')> \\
&&+<c(y),\nabla_xp(t,x,y,x',y')>,
\quad (t,x,y)\in\R_+^*\times\R^{\tilde d}\times\R^d\\
 p(0,x,y,x',y')&=& \delta_{x'}(x)\otimes\delta_{y'}(y),\quad (x,y)\in\R^{\tilde d}\times\R^d,
 \end{array}\right.
\end{equation}
as a function of the starting point $(x,y)$.
Equations of type (\ref{Kolm-back})  were first studied by Kolmogorov in \cite{Ko:Ko}. Since then, such equations are known as Kolmogorov hypoelliptic equations (KHE). 
Note however that in \cite{Ko:Ko} the KHE appears in a time-inhomogeneous forward form, while Equation 
\eqref{Kolm-back} is time-homogeneous and in the backward form (for more details on the {\it backward/forward} terminology, see the appendix Section \ref{rem:back-for}).

One of the objectives of this paper is to provide a probabilistic approach to compute the solution to KHEs described by \eqref{Kolm-back}, and thus to compute the transition probability function of SDEs described by \eqref{hypo1} by identification using \eqref{eq:transition}.
KHEs described by \eqref{Kolm-back} can be considered as particular cases of equations described  by \eqref{Kolm-fund} with $\alpha=0$. In practice, we may use the solution to an equation of type 
\eqref{Kolm-fund}, with $\alpha\neq 0$ a complex number, in order to get a solution to \eqref{Kolm-back}, using Fourier transform arguments as follows.
In the following, we define the Fourier transform w.r.t the $x$ variable 
as~$\hat \phi(\gamma)=\int_{\R^{\tilde d}} e^{-i<\gamma ,x>}\phi(x)dx$. Then, taking the Fourier transform w.r.t the $x$ variable in~\eqref{Kolm-back} yields
\begin{equation}\label{Kolm1}
\left\{ \begin{array}{rcl}
\partial_t {\hat p}(t,\gamma,y,x',y')&=&\frac12\Delta_y\hat p(t,\gamma,y,x',y')+<b(y),\nabla_y \hat p(t,\gamma,y,x',y')>\\
&& +i<\gamma, c(y)> \hat p(t,\gamma,y,x',y'), \quad(t,\gamma,y)\in \R_+^*\times\R^{\tilde d}\times\R^d  \\
 \hat p(0,\gamma,y,x',y')&=&e^{-i<\gamma, x'>}\delta_{y'}(y),\, (\gamma,y)\in \R^{\tilde d}\times\R^d.
 \end{array}\right.
\end{equation}
For fixed Fourier variable $\gamma$, Equation \eqref{Kolm1} is similar to Equation \eqref{Kolm-fund} with $\tilde{d}=0$, $\alpha = i$ and $V_{\gamma}(y)=<\gamma, c(y)>$.
Note that the initial condition has been replaced by $p(0,\gamma,y,x',y')=e^{-i<x',\gamma>}\delta_{y'}(y)$, which is not a major issue for solving \eqref{Kolm1}.
Then, taking the inverse Fourier transform 
$\phi(x)=\frac{1}{2\pi}\int_{\R^{\tilde d}} e^{i<\gamma ,x>}\hat{\phi}(\gamma)\,d\gamma$ of the solution to \eqref{Kolm1}, we will deduce the solution to KHEs of type \eqref{Kolm-back}.

Our study is in close connection with
\cite{Ca:Ca}, in which the authors
also use the Fourier transform method. However, their analysis is then based on  a semi-classical  approximation ``à la Morette-DeWitt'' \cite{DeW:DeW}. Note that the results in \cite{Ca:Ca} have been deeply expanded in \cite{Ca1:Ca1}. The approach we propose in the present paper is based on F-K formula leading to the expectation of a Brownian motion functional which is then computed exactly or approximately. 
For the computation of this last term, we resort to a regression model between the Brownian motion $\{W(s),\,0\le s<t\}$ and the value at terminal time $W(t)$. This procedure is a well-known tool in mathematical physics (see, e.g., \cite[Theorem 6.6]{Si:Si}).

This note is intended to introduce a topic, well known to analysts, to a probabilistic audience. Some of the results obtained are known, others are new. In all the results we claim some originality in the procedures and how simple the proofs are. A similar approach has been applied in \cite{Ca2:Ca2} to degenerated elliptic operators. Nevertheless, there exist remarkable differences with our work in the computations. Moreover, we propose in Section \ref{smallt} a new result on {\color{black}  an approximation in small time} of the solution to the KHE \eqref{Kolm-fund}, {\color{black} in the case $b\equiv 0$ and $\alpha=0$}.

{\color{black}Note that the equations we handle in this paper, although specific, appear in different fields of application, such as, e.g., finance or physics. We provide below two examples extracted from recent literature.
The first one is mentioned in Calin et al. \cite{Ca:Ca}. Let $\tilde d =1$ and $d=2$ and consider the KHE with $b \equiv 0$ and $c(y)=y_1-y_2$ with $y_1$ and $y_2$ the components of $y$. Then the following equation
\begin{equation*}
\left\{\begin{array}{lcl}
\frac{\partial u}{\partial t}&=&\Delta_{y}u+(y_1-y_2)\frac{\partial u}{\partial x}\\
u(0,x,y)&=&f(x,y)\end{array}\right.
\end{equation*}
governs the pricing of options on geometric moving averages.
The second example, from Bian et al. \cite{Bian:Bian}, consists in the Fokker-Planck equation governing the evolution of the phase-space distribution  of photons, $N(y,x,t)$: 
\begin{equation}\label{fokpho}
\partial_t N+\frac{c^2}{\omega} k\cdot \nabla_x N=\frac1{2 \overline \omega}\nabla_x\overline\omega^2_{pe}\cdot\nabla_y N+\sum_{ij}\partial_{y_i}(a_{ij}(y)\partial_{y_j} N),
\end{equation}
with $(x,y,t) \in \mathbb{R}^3 \times \mathbb{R}^3 \times \mathbb{R}_+^*$. We refer to \cite{Bian:Bian} for the definition of the constants appearing in the above equation. 
In \cite{Bian:Bian} different limit regimes for \eqref{fokpho} are studied, giving rise to different reduced systems.
In the small-angle approximation, Equation \eqref{fokpho} reduces to the following Kolmogorov linear hypoelliptic equation (see Eq. (80) in \cite[Section 4]{Bian:Bian}):
$$\frac{\partial p}{\partial t}=\frac12\frac{\partial^2 p}{\partial y^2}-y\frac{\partial p}{\partial x},$$
with $p$ the probability density function (pdf) characterizing the perpendicular dynamics of photons.
Besides, the analysis of the diffusive regime is given in \cite[Section 5]{Bian:Bian}, it consists in the analysis of the spatial dispersion along the $x_3$-axis. Their study leads to the computation of the pdf $Q(T,\tau)$ of the delay time $T$ at time $\tau$, defined as $T(\tau)=\int_0^\tau W^2(s)ds$, with $W$ the standard Brownian motion. More precisely, they state that the pdf $p(\tau,x,y)$ of the vectorial diffusion $(T(\cdot),W(\cdot))$ is governed by:
$$\frac{\partial p}{\partial\tau}=\frac12\frac{\partial^2p}{\partial y^2}-y^2\frac{\partial p}{\partial x},$$ with some initial condition. It corresponds to the KHE with $\tilde d = d =1$, $b \equiv 0$ and $c(y)=-y^2$. One quantity of interest in \cite[Section 5]{Bian:Bian} is then the Fourier transform $\hat p (\tau, \gamma, y)=\int e^{-\gamma x}p(\tau,x,y)dx$, which can be computed using our results in Section \ref{KHEsection}.}


Our paper is organized as follows. In Section \ref{sec:FK}, we recall the so-called F-K formula and we explain how it can be used to compute the fundamental solution to parabolic equations with linear or quadratic potential. Then we deduce from these results the computation of the solution to a first class of KHEs. In Section \ref{sec:OU}, we solve partial differential equations obtained via the F-K formulas from the Ornstein-Uhlenbeck generator. We propose in Section \ref{smallt} a new small time approximation of the solution to KHEs ({\color{black} case $b\equiv 0$ and $\alpha=0$}). We finally compare in Section \ref{sec:num} the numerical approximation we propose with other ones from the literature. The appendix section provides more details on the connexion between the solution to KHEs in the backward or forward form and the transition probability function of processes governed by SDEs.

\section{Feynman-Kac formula, notation and first examples}
\label{sec:FK}

In this section, we recall how a link between partial differential equations and stochastic processes can be established by using the so-called Feynman-Kac (F-K) formula. Then in Section \ref{examples} we exploit this link to compute the fundamental solution to parabolic equations with linear or quadratic potential. In Section \ref{KHEsection} we apply these results to the computation of the solution to a first class of Kolmogorov hypoelliptic equations.
F-K formula originally was introduced as a tool to compute the solution of certain parabolic partial differential equations as the expectation of some Brownian motion functional (see for instance \cite{Sch:Sch}). 

Let us introduce some notation. The process  $B(t)=y+W(t)$ denotes the Brownian motion beginning at $y$. The expectation with respect to this last process will be denoted by $\E^y$ and we set $\E=\E^0$. We consider the following SDE:
 \begin{equation}\label{eq:simple}
 dB(t)=dW(t) \quad \textup{ with } \quad B(0)=y.
 \end{equation}
 
 Let $V:\R^d\to\R$ be a continuous  function such that \begin{eqnarray}\label{hipotesis}\E[e^{\beta \int_0^t V(y+W(s))ds}]<\infty  \textup{ for some } \beta\in\R \textup{ and for all }t \in \mathbb{R}_+ \textup{ and } y\in\R^d. \end{eqnarray} 
 Then, if $\alpha=\beta+i\gamma$, the following semigroup 
$\displaystyle P_t^Vf(y)=\E^y[e^{\alpha\int_0^tV(B(s))ds}f(B(t))]$
 is a family of continuous operators acting on the space $\mathbf C_b(\R^d;\mathbb C)$ of continuous and bounded functions taking complex values. It is an easy matter to prove that the semigroup also acts continuously on $\mathbb L^2(\R^d)$. By defining $u(t,y)=P_t^Vf(y)$, the celebrated F-K formula establishes that this function is the solution of the following partial differential equation:
\begin{equation}
\label{Ham}
\left\{\begin{array}{rcll}
\displaystyle \frac{\partial u}{\partial t}(t,y)&=&\frac12\Delta u(t,y)+\alpha V(y)u(t,y),&\forall t>0,\\\\
u(0,y)&=&f(y)&\\
\end{array}\right.
\end{equation}
which {\color{black}consists in} a particular case of partial differential equations described by \eqref{Kolm} with $\tilde d =0$ and $b \equiv 0$.

\begin{Remark} \label{nota1}There exist different conditions under which Hypothesis (\ref{hipotesis}) holds. For instance if $\beta < 0$ then it is enough that $V$ is bounded by below and if $\beta>0$ a simple condition is $|V(y)|\le C \left( 1+||y||^2\right)$. In the following we assume one of the two conditions according to the case. {\color{black}Note that in this section, as far as in the rest of the paper, for any $k \geq 1$ and any $v \in \mathbb{R}^k$, $\|v\|$ denotes the euclidean norm on $\mathbb{R}^k$.}\end{Remark}

To conclude this section, let us compute, by using F-K formula, the fundamental solution 
$p^V(t,y,z)$ of Equation (\ref{Ham}). Let $\displaystyle p_t(z)=(2\pi t)^{-\frac{d}{2}} e^{-\frac{||z||^2}{2t}}$ denote the probability density function of $W(t)$. Then, using \eqref{eq:rep-sto} and the total probability theorem, we can write

\[
\begin{array}{rcl}
u(t,y) & = & \int_{\R^d} p^V(t,y,z) f(z)dz
=\E^y[e^{\alpha\int_0^tV(B(s))ds}f(B(t))]\\
&=&\int_{\R^d}\E[e^{\alpha\int_0^tV(y+W(s))ds}|W(t)=z]p_t(z)f(y+z)dz\\
&=&\int_{\R^d}\E[e^{\alpha\int_0^tV(y+W(s))ds}|W(t)=z-y]p_t(z-y)f(z)dz.
\end{array}
\]

We deduce from the above equality
\begin{eqnarray} \label{fundamental} p^V(t,y,z)=\E[e^{\alpha\int_0^tV(y+W(s))ds}|W(t)=z-y]p_t(z-y).\end{eqnarray}

For specific potentials $V$, it is possible to derive from \eqref{fundamental} analytical expressions of~$p^V(t,y,z)$. Two important potentials verifying  (\ref{hipotesis}) are: $V(y)=<a,y>$ for $a\in\R^d$ and~$V(y)=\frac12||\Omega^{\frac12}y||^2$ with $\Omega$ a diagonalizable symmetric matrix with non zero eigenvalues. Both cases will be handled in Section \ref{examples}. Formula \eqref{fundamental} can also be used for obtaining an asymptotic expansion of the fundamental solution for small $t$ (see Section~\ref{smallt}).

\subsection{Fundamental solution to parabolic equations with linear or quadratic potential}\label{examples}

 In this section, we provide the analytical expression of the fundamental solution to the partial differential equation described by \eqref{Ham} for a linear potential in Proposition \ref{scalar} (see \cite[Example 2]{Gzy:Leo}), then for a quadratic potential in Proposition \ref{prop:quad-multidim}.

\begin{proposition}\label{scalar}
Let $V(y)=<a,y>$, with $a\in\R^d$. For any $\alpha\in\C$ the fundamental solution of \eqref{Ham} is
\begin{equation*}
\label{eq:lin-multidim}
    p^V(t,y,z)=e^{\frac{\alpha t}2<a,z+y>}e^{||a||^2\frac{\alpha^2}2\sigma^2_\xi(t)}p_t(z-y),
\end{equation*}
with $\sigma^2_\xi(t)=t^3/12$.
\end{proposition}

\begin{proof}
From \eqref{fundamental} we have the following formula:
\begin{eqnarray} 
\nonumber p^V(t,y,z) & = & \E[e^{\alpha\int_0^t<a,y+W(s)>ds}|W(t)=z-y]p_t(y-z)\\
\nonumber & = & e^{\alpha t<a,y>}\E[e^{\alpha||a||\int_0^t<\frac{a}{||a||},W(s)>ds}|W(t)=z-y]p_t(y-z).
\end{eqnarray}

Let $P_1$ be the rotation of axes such that $\frac{a}{||a||}=P_1(e_1)$, with $e_1$ the first coordinate vector. This is a unitary transformation and the process $\tilde W(\cdot)=P^{-1}_1W(\cdot)$ is also a standard Brownian motion, by Lévy theorem. Then we can write
$$\begin{array}{lcl}
\nonumber \E[e^{\alpha\int_0^t<a,W(s)>ds}|W(t)=z-y]  =  \E[e^{\alpha||a||\int_0^t<e_1,P^{-1}_1W(s)>ds}|P^{-1}_1W(t)=P^{-1}_1(z-y)]&&
\\\\
\nonumber \quad = \E[e^{\alpha||a||\int_0^t<e_1,\tilde W(s)>ds}|\tilde W(t)=P^{-1}_1(z-y)]=\E[e^{\alpha||a||\int_0^t\tilde W_1(s)ds}|\tilde W_1(t)=<e_1,P^{-1}_1(z-y)>]&&\\\\
\nonumber \quad =  \E[e^{\alpha||a||\int_0^t\tilde W_1(s)ds}|\tilde W_1(t)=<\frac{a}{\|a\|},z-y>].&&
\end{array}$$ 

In this form we have reduced our problem to the one-dimensional one. 
Then we use the following regression model. Let us define $Z(t)=\int_0^t \tilde W_1(s)ds$  a zero mean Gaussian random variable. Thus we can write the regression of $Z(t)$ on $\tilde W_1(t)$ involving $\xi(t)$, another zero mean Gaussian random variable:
$$Z(t)=\zeta(t) \tilde W_1(t)+\xi(t),\quad \xi(t)\perp \tilde W_1(t).$$
We have $\displaystyle \zeta(t)=\frac{\mathbb{E}[Z(t)\tilde W_1(t)]}{\mathbb{E}[\tilde W_1^2(t)]}=\frac t2$. Then $\E[\xi^2(t)]=\frac{t^3}4-\E[Z^2(t)]$, thus $\sigma^2_\xi(t)=\frac{t^3}{12}.$
It yields:
$$
p^V(t,y,z)=e^{\frac{\alpha t}2<a,z+y>}\E[e^{||a||\alpha\xi(t)}]p_t(y-z),$$
using the moment-generating function of normal law we get
\begin{equation*} \label{onedim}p^V(t,y,z)=e^{\frac{\alpha t}2<a,z+y>}e^{\frac{||a||^2\alpha^2\sigma_\xi^2(t)}2}p_t(y-z).\end{equation*}

\end{proof}

\begin{proposition}
\label{prop:quad-multidim}
Let $V(y)=\frac12\,||\Omega^{\frac12}y||^2$, where $\Omega=P^{-1}DP$ with $PP^T=I_d$ and $D$ a diagonal matrix with  non zero real coefficients. For any $\alpha\in\C$ the fundamental solution of \eqref{Ham} is
\begin{equation}\label{multoscill}p^V(t, y,z)=\prod_{i=1}^d\Big[\frac{\sqrt{\alpha\rho_i}t}{\sin(\sqrt{\alpha\rho_i}t)}\Big]^{\frac12}
e^{-S(t,x(\cdot))}p_t(0),\end{equation}
where the $\rho_i$'s, $1\leq i\leq d$, are the eigenvalues of $D$, the function $x(\cdot):[0,t]\to\C^d$ solves 
\begin{equation}
    \label{eq:x-multidim}
    x''=-\alpha Dx,
\end{equation}
with $x(0)=Pz$, $x(t)=Py$, and $S(t,\cdot)$ is the action functional defined by
\begin{equation}
    \label{eq:S-multidim}
    S(t,\gamma)=\frac 1 2 \Big[ \int_0^t||\gamma'(s)||^2ds-\alpha\int_0^t||D^{1/2}\gamma(s)||^2ds  \Big],
\end{equation}
for any smooth $\gamma:[0,t]\to\C^d$.
\end{proposition}

\begin{Remark}
In Proposition \ref{prop:quad-multidim}, note that $\Omega^{\frac12}=P^{-1}D^{1/2}P$ and that $||\cdot||$ denotes the norm induced on $\mathbb{C}^d$ by the usual hermitian product, as the coefficients in $D^{1/2}$ are possibly complex numbers.
\end{Remark}

\begin{proof}[Proof of Proposition \ref{prop:quad-multidim}]
Let $u$ be the solution of \eqref{Ham}. Using the (multidimensional-complex) Cameron-Martin-Girsanov formula \cite{Az:Doss} we get, for any smooth function $\psi$ satisfying $\psi(0)=0$,
$$\begin{array}{rcl}
u(t,y) & = & \E^y[e^{\frac\alpha2\int_0^t\|\Omega^{\frac12}B(s)\|^2ds} f(B(t))]\\
& =& \E[e^{\frac\alpha2\int_0^t\|\Omega^{\frac12}(\psi(s)+y+W(s))\|^2ds} e^{-\int_0^t \psi'(s)dW(s)-\frac12\int_0^t\|\psi'(s)\|^2ds}f(\psi(t)+y+W(t))]\\
& = &\int_{\R}\E[e^{\frac\alpha2\int_0^t\|\Omega^{\frac12}(\psi(s)+y+W(s))\|^2ds} e^{-\int_0^t \psi'(s)dW(s)-\frac12\int_0^t\|\psi'(s)\|^2ds}|W(t)=z-\psi(t)-y]\\
& & \qquad \qquad \qquad f(z)p_t(z-\psi(t)-y)dz.
\end{array}$$

Further, choosing $\psi$ s.t. $\psi(t)=z-y$ and using time inversion arguments, we get
$$
p^{V}(t,y,z)=\E[e^{\frac\alpha2\int_0^t||\Omega^{\frac12}(\psi(t-s)+y+W(t-s))||^2ds} e^{\int_0^t \psi'(t-s)dW(t-s)-\frac12\int_0^t||\psi'(t-s)||^2ds}|W(t)=0]p_t(0).
$$
We set now $x(s)=P\psi(t-s)+Py$. Noticing that $\psi'(t-s)=-P^{-1}x'(s)$, it holds
$$
p^{V}(t,y,z)=\E[e^{\frac\alpha2\int_0^t||P^{-1}D^{\frac12}(x(s)+PW(t-s))||^2ds} e^{-\int_0^t P^{-1}x'(s)dW(t-s)-\frac12\int_0^t||P^{-1}x'(s)||^2ds}|W(t)=0]p_t(0).
$$
Using now successively $\displaystyle \left( W(\cdot)|\{W(t)=0\}\right) \stackrel{d}= \left(W(t-\cdot)| \{W(t)=0\}\right)$ ($0 \leq s \leq t$) and 
$PP^T=I_d$ it follows
$$
\begin{array}{lll}
p^V(t,y,z)&=&
\E[e^{\frac\alpha2\int_0^t||P^{-1}D^{\frac12}(x(s)+PW(s))||^2ds} e^{-\int_0^t P^{-1}x'(s)dW(s)-\frac12\int_0^t||P^{-1}x'(s)||^2ds}|W(t)=0]p_t(0)\\
\\
&=&\E\big[e^{\frac\alpha2\int_0^t||D^{\frac12}PW(s))||^2ds} 
e^{\frac\alpha2\int_0^t||D^{\frac12}x(s))||^2ds
+\int_0^t\alpha<D^{\frac12}x(s),D^{\frac12}PW(s)>ds}\\
\\
&&\hspace{1cm}e^{-\int_0^t x'(s)d\,PW(s)-\frac12\int_0^t||x'(s)||^2ds}|W(t)=0\big]p_t(0).\\
\end{array}
$$

Thus, considering the new standard Brownian motion 
$\tilde{W}=PW$ we have
\begin{equation}
    \label{eq:multidim1}
    \begin{array}{lll}
p^V(t,y,z)&=&\E\big[e^{\frac\alpha2\int_0^t||D^{\frac12}\tilde{W}(s))||^2ds} 
e^{\frac\alpha2\int_0^t||D^{\frac12}x(s))||^2ds
+\int_0^t\alpha<Dx(s),\tilde{W}(s)>ds}\\
\\
&&\hspace{1cm}e^{-\int_0^t x'(s)d\tilde{W}(s)-\frac12\int_0^t||x'(s)||^2ds}|\tilde{W}(t)=0\big]p_t(0).\\
\end{array}
\end{equation}

We now assume that $x(\cdot)$ satisfies \eqref{eq:x-multidim}
and recall that  $x(0)=Pz$ and $x(t)=Py$ (note that Eq. \eqref{eq:multidim1} is satisfied for any $x(\cdot)$ satisfying 
\eqref{eq:x-multidim} and $x(0)=Pz$, $x(t)=Py$). Under the conditioning 
$\tilde{W}(t)=\tilde{W}(0)=0$,
we have
$$
\int_0^t x'(s)d\tilde{W}(s)=\alpha\int_0^t<Dx(s),\tilde{W}(s)>ds,
$$
so that
\begin{eqnarray}\label{quadratic1}
p^V(t,y,z)=\E\big[e^{\frac\alpha2\int_0^t||D^{\frac12}\tilde{W}(s))||^2ds}|\tilde{W}(t)=0\big]
e^{-S(t,x(\cdot))}p_t(0),\end{eqnarray}
where $S(t,\cdot)$ is the action functional defined by \eqref{eq:S-multidim}.
We now turn to the computation of 
\begin{equation*}
\E\big[e^{\frac\alpha2\int_0^t||D^{\frac12}\tilde{W}(s))||^2ds}|\tilde{W}(t)=0\big]
=\E[e^{\alpha\sum_{i=1}^d\int_0^t\rho_i\tilde W^2_i(s)ds}|\tilde W(t)=0].
\end{equation*}

Using the independence between the coordinates of the Brownian motion $\tilde W$ we get 
\begin{eqnarray*}\E\big[e^{\frac\alpha2\int_0^t||D^{\frac12}\tilde{W}(s))||^2ds}|\tilde{W}(t)=0\big] =
\prod_{i=1}^d
\E[e^{\frac\alpha2\sum_{i=1}^d\int_0^t\rho_i\tilde W^2_i(s)ds}|\tilde W_i(t)=0]
.\end{eqnarray*}

The conditional expectation for each term of the product can be computed using the following regression model 
$$\tilde W_i(s)=\frac st \tilde W_i(t)+Z_i(s),$$ where $Z_i(\cdot)$ is independent of $\tilde W_i(t)$. In this form, we have that $Z_i(\cdot)$ is a mean zero Gaussian process with covariance function
$\E[Z_i(s_1)Z_i(s_2)]=s_1\wedge s_2-\frac{s_1s_2}t$. This process satisfies $Z_i(t\cdot)\stackrel d=\sqrt t b_i(\cdot),$  where $b_i$ are independent  Brownian bridges. Thus
\begin{eqnarray}\label{regression3}
\nonumber \E[e^{\frac{\alpha\rho_i}2\int_0^t\tilde W_i^2(s)ds}|\tilde W(t)=0)]=\E[e^{t^2\frac{\alpha\rho_i}2\int_0^1b_i^2(s)ds}].
\end{eqnarray}

It is known that the Brownian bridge admits the following representation
\begin{eqnarray}\label{expansion}b_i(s)=\sqrt2\sum_{k=1}^\infty z_{k,i}\frac{\sin(k\pi s)}{k\pi},\end{eqnarray}where $z_{k,i}$ is a sequence of $\mathcal N(0,1)$ independent random variables.
Then, by using expansion (\ref{expansion}) and the moment generating function of the $\chi^2_1$, we get:
\begin{equation*}\label{eq:apres}
\E[e^{t^2\frac{\alpha\rho_i}2\int_0^{1}b^2(s)ds}]=\prod_{k=1}^\infty \E[e^{\frac{t^2\frac{\alpha\rho_i}2}{\pi^2 k^2}z^2_{k,i}}]=\prod_{k=1}^\infty \frac1{(1-\frac{t^2\alpha\rho_i}{\pi^2 k^2})^{\frac12}}=\left(\frac{\sqrt {\alpha\rho_i} t}{\sin(\sqrt {\alpha\rho_i} t)}\right)^{\frac12}.
\end{equation*}

Here we have used the Weierstrass-Hadamard factorization formula of the  sine  function at the last inequality. 
Therefore from (\ref{quadratic1}) we obtain that \eqref{multoscill} holds and the proof is completed.
\end{proof}
\begin{Remark}
Our procedure is inspired in the seminal work of Azencott \& Doss \cite{Az:Doss}, we refer the reader to this paper for  more details. The potential $V(y)=y^2/2$ corresponds to the Hamiltonian for the quantum harmonic oscillator.
For a deep and ingenious insight into this last computation see \cite[p. 72-73]{fey:Hibs}.
The results in Proposition \ref{scalar} and in Proposition~\ref{prop:quad-multidim} are well known in dimension $d=1$. If we take $\alpha\in\R_-^*$ and $d=1$ these are the results obtained in \cite{Ca:Ca}, using the already quoted method of \cite{DeW:DeW}. If $\alpha\in\R_+^*$ and $d=1$ these are the results announced in \cite[Problems 3-8 and 3-9]{fey:Hibs}. 
\end{Remark}

\subsection{Application to the computation of the solution to a first class of Kolmogorov hypoelliptic equations} \label{KHEsection}

In this section, we focus on the Kolmogorov hypoelliptic equation (KHE), that is we set~$\alpha =0$ in  \eqref{Kolm-back}. We consider here two examples with $\tilde d =1$ and $d \geq 1$, with $b\equiv 0$ and a linear or quadratic coefficient $c(y)$. The results are deduced from Propositions \ref{scalar} and \ref{prop:quad-multidim} after a Fourier transform with respect to the variable $x$.

\bigskip

\noindent {\bf Linear case:} the equation has the form
$$\left\{ 
\begin{array}{l}
\displaystyle \partial_tp(t,x,y,x',y')=\frac12\Delta_y p(t,x,y,x',y')-<a,y>\partial_xp(t,x,y,x',y'),\\
\hspace{7cm} (t,x,y)\in\R^*_+\times\R\times\R^d\\\\
\displaystyle p(0,x,y,x',y')=\delta_{x'}(x)\otimes\delta_{y'}(y),\qquad  (x,y)\in \R\times \R^d.
\end{array}\right.$$
Let us denote its solution by $p^{<a,y>}(t,x,y,x',y')$.
Then, taking Fourier transform with respect to the variable $x \in \mathbb{R}$ we get for any $\gamma\in\R$
$$\left\{\begin{array}{l}
\displaystyle \partial_t\hat{p}(t,\gamma,y,x',y')=\frac12\Delta \hat p(t,\gamma,y,x',y')-i\gamma<a,y>\hat p(t,\gamma,y,x',y'),\\
\hspace{7cm} (t,y)\in\R^*_+\times\R^d\\
\\
\displaystyle \hat p(0,\gamma,y,x',y')=e^{-i\gamma x'}\delta_{y'}(y), \quad y\in\R.
\end{array}\right.$$

 Applying Proposition \ref{scalar} with $\alpha=-i\gamma$ and then the inverse Fourier transform we get the following classical result (see \cite{Ko:Ko} for the case $d=1$).
 
 \begin{coro}\label{cor:linearcase}We have
$$\displaystyle p^{<a,y>}(t,x,y,x',y')=  \frac{e^{-\frac{(x'-x-t\frac{<a,y+y'>}2)^2}{2||a||^2\sigma^2_\xi(t)}}}{\sqrt{2\pi}||a||\sigma_\xi(t)}p_t(y'-y)$$
with $\sigma^2_\xi(t)=t^3/12$.
\end{coro}

\bigskip

\noindent {\bf Quadratic case:} 
the KHE has the form
$$\left\{
\begin{array}{l}
\displaystyle \partial_tp(t,x,y,x',y')=\frac12\Delta_y p(t,x,y,x',y')+||D^{\frac12}y||^2\partial_xp(t,x,y,x',y')\\\\
\displaystyle p(0,x,y,x',y')=\delta_{x'}(x)\otimes\delta_{y'}(y),
\end{array}\right.$$
with $D$ is a positive definite diagonal matrix with eigenvalues $\rho_i\neq0$. The quadratic potential plays an important role in the study  of the harmonic oscillator. Let us denote by $p^{||D^{\frac12}y||^2}(t,x,y,x',y')$ the solution of this equation.
As previously, we take the Fourier transform with respect to the variable $x$ and the equation becomes
$$\left\{
\begin{array}{l}
\displaystyle \partial_t\hat{p}(t,\gamma,y,x',y')=\frac12\Delta_y \hat{p}(t,\gamma,y,x',y')+i\gamma ||D^{\frac12}y||^2\hat{p}(t,\gamma,y,x',y')\\\\
\displaystyle \hat p(0,\gamma,y,x',y')=e^{-i\gamma x'}\delta_{y'}(y).
\end{array}\right.$$
Applying Proposition \ref{prop:quad-multidim} with $\alpha=2i\gamma$, and then the inverse Fourier transform we will get the following result.
\begin{coro}
We have
$$p^{||D^{\frac12}y||^2}(t,x,y,x',y')=(\mathcal P_{i=1}^du(t,\cdot,y_i,y'_i))(x-x'),$$
where we have denoted by
$$\mathcal P_{i=1}^d g_i(x)=g_1*g_2*\ldots* g_d(x)$$
the convolution product between $d$ functions $\{g_i\}_{i=1}^d$, and where 
\begin{eqnarray}\label{Inverse}
u(t,x,y_i,y'_i)=\frac1{2\pi}\int_\R e^{i\gamma x}f(\gamma,\rho_i,y_i,y'_i)d\gamma\end{eqnarray}
with
\begin{equation}
\label{eq:def-f}
f(\gamma,\rho,y_i,y'_i)=\frac1{\sqrt{2\pi}}\Big[\frac{\sqrt{2i\gamma\rho}}{\sin(\sqrt{2i\gamma\rho}t)}\Big]^{\frac12}e^{-\frac12\frac{\sqrt{-2i\gamma\rho}}{\sinh(\sqrt{-2i\gamma\rho}t)}(((y'_i)^2+y^2_i)\cosh(\sqrt{-2i\gamma\rho}t)-2y'_iy_i)},\;\;\forall 1\leq i\leq d.
\end{equation}
\end{coro}

\begin{proof}
Applying Proposition \ref{prop:quad-multidim} we get 
\begin{equation*}\hat{p}(t,\gamma,y,x',y')=e^{-i\gamma x'}\prod_{i=1}^d\Big[\frac{\sqrt{2i\gamma\rho_i}t}{\sin(\sqrt{2i\gamma\rho_i}t)}\Big]^{\frac12}
e^{-S(t,x(\cdot))}p_t(0)\end{equation*} 
where $x:[0,t]\to\R^d$ is solution of the equation
$$x''=-2i\gamma Dx,$$ with boundary conditions $x(0)=y'$ and $x(t)=y$ and where
$$\begin{array}{l}
\displaystyle S(t,x(\cdot)))\\\\
\displaystyle \quad =\frac12\sum_{i=1}^d\big[\int_0^t x'_i(s)x'_i(s)ds-2i\gamma\rho_i\int_0^t x^2_i(s)ds\big]=\frac12\sum_{i=1}^d(x'_i(t)x_i(t)-x'_i(0)x_i(0))\\\\
\displaystyle \quad =\frac12\sum_{i=1}^d\frac{\sqrt{-2i\gamma\rho_i}}{\sinh(\sqrt{-2i\gamma\rho_i}t)}(((y'_i)^2+y^2_i)\cosh(\sqrt{-2i\gamma\rho_i}t)-2y'_iy_i).
\end{array}$$

Noting that
$$\begin{array}{l}
\displaystyle e^{-i\gamma x'}\prod_{i=1}^d\Big[\frac{\sqrt{2i\gamma\rho_i}t}{\sin(\sqrt{2i\gamma\rho_i}t)}\Big]^{\frac12}
e^{-S(t,x(\cdot))}p_t(0)\\\\
\displaystyle \quad =\frac1{(2\pi)^{\frac d2}}e^{-i\gamma x'}\prod_{i=1}^d\Big[\frac{\sqrt{2i\gamma\rho_i}}{\sin(\sqrt{2i\gamma\rho_i}t)}\Big]^{\frac12}e^{-\frac12\frac{\sqrt{-2i\gamma\rho_i}}{\sinh(\sqrt{-2i\gamma\rho_i}t)}(((y'_i)^2+y^2_i)\cosh(\sqrt{-2i\gamma\rho_i}t)-2y'_iy_i)}.
\end{array}$$
and applying the inverse Fourier transform we get  
$$\begin{array}{l} 
\displaystyle p^{||D^{\frac12}y||^2}(t,x,y,x',y')\\\\
\displaystyle \quad =\frac1{2\pi}p_t(0)\int_\R e^{i\gamma(x-x')}\prod_{i=1}^d\Big[\frac{\sqrt{2i\gamma\rho_i}t}{\sin(\sqrt{2i\gamma\rho_i}t)}\Big]^{\frac12}
e^{-S(t,x(\cdot))}d\gamma\\\\
\displaystyle \quad =\frac1{(2\pi)^{\frac d2+1}}\int_\R e^{i\gamma(x-x')}\prod_{i=1}^d\Big[\frac{\sqrt{2i\gamma\rho_i}}{\sin(\sqrt{2i\gamma\rho_i}t)}\Big]^{\frac12}e^{-\frac12\frac{\sqrt{-2i\gamma\rho_i}}{\sinh(\sqrt{-2i\gamma\rho_i}t)}(((y'_i)^2+y^2_i)\cosh(\sqrt{-2i\gamma\rho_i}t)-2y'_iy_i)} d\gamma.
\end{array}$$
Defining $f(\gamma,\rho,y_i,y'_i)$ as in \eqref{eq:def-f}
 we have
$$p^{||D^{\frac12}y||^2}(t,x,y,x',y')=\frac1{2\pi}\int_\R e^{i\gamma(x-x')}\prod_{i=1}^d f(\gamma,\rho_i,y_i,y'_i)d\gamma.$$
Defining $u(t,x,y_i,y'_i)$ as in \eqref{Inverse}
we get
$$p^{||D^{\frac12}y||^2}(t,x,y,x',y')=(\mathcal P_{i=1}^du(t,\cdot,y_i,y'_i))(x-x').$$
Hence, the computation of $p^{||D^{\frac12}y||^2}(t,x,y,x',y')$ amounts to the computation of the integral appearing in (\ref{Inverse}).
\end{proof}

\begin{Remark}
Note that in \cite[Theorem 3.4]{Ca:Ca}, a similar formula was obtained (for $\alpha=\gamma\in\R$) also expressed by means of an integral of a complex valued function that is not explicitly computable.
\end{Remark}


\section{Ornstein-Ulhenbeck generator}\label{sec:OU}
In this section we solve partial differential equations obtained via F-K formulas from the Ornstein-Uhlenbeck generator. Once more the partial differential equations we consider are particular case of \eqref{Kolm}.

\subsection{Fundamental solution to the Ornstein-Uhlenbeck parabolic equation with linear potential}
In all this section, we consider $d=1$ and $\tilde d =0$.
Let us consider the Ornstein-Uhlenbeck process, defined as the solution of the stochastic differential equation
\begin{equation}\label{OU}
dY(t)=dW(t)-\zeta Y(t)dt \, , \quad Y(0)=w.
\end{equation}

The infinitesimal generator of this process is 
$\displaystyle \frac12\frac{\partial^2}{\partial^2 y}-\zeta y\frac{\partial}{\partial y}.$
We know from F-K formula that function 
$u(t,y)=\E^y[e^{\alpha\int_0^t V(Y(s))ds} f(Y(t))]$ is solution of the following partial differential equation
$$\left\{
\begin{array}{l}\label{OUeq}
   \displaystyle \frac{\partial u}{\partial t}=\frac12\frac{\partial^2u}{\partial^2 y}-\zeta y\frac{\partial u}{\partial y}+\alpha V(y) u\\\\
    \nonumber  u(0,y)=f(y).
\end{array}\right.$$

We aim in this section at computing an explicit formula for the fundamental solution of the above equation that we denote by $p_{OU}^V$. To easier the computations, we focus on the linear case $V(y)=y$. Our result is stated in Proposition \ref{prop:OU}.

\begin{proposition}\label{prop:OU}
Let $V(y)=y$. We have, for any $\alpha\in\C$,
\begin{equation} \label{OUFS}
p_{OU}^V(t,y,z)=e^{\alpha\frac{y}{\zeta}(1-e^{-\zeta t})}e^{\frac{\alpha}{\zeta}(z-ye^{-\zeta t})\tanh(\frac\zeta2t)}e^{\frac12\alpha^2\sigma^2_{\xi(t)}}\sqrt{\frac\zeta\pi}\frac{e^{-\frac{\zeta(z-ye^{-\zeta t})^2}{(1-e^{-2\zeta t})}}}{\sqrt{1-e^{-2\zeta t}}}
\end{equation}
with $\sigma^2_{\xi(t)}=\sigma^2_{Z(t)}-\frac1{2\zeta^3}\frac{(1-e^{-\zeta t})^3}{(1+e^{-\zeta t})}$ and 
$\sigma^2_{Z(t)}=\frac1{\zeta^2}\int_0^t(1-e^{-\zeta(t-u)})^2du$.
\end{proposition}

\begin{proof}
We have
$$\int_\R p_{OU}^V(t,y,z)f(z)dz=\E^y[e^{\alpha\int_0^tY(s)ds}f(Y(t))]
=\int_\R \E^y[e^{\alpha\int_0^tY(s)ds}|Y(t)=z]f(z)p_t^{Y(t)}(z)dz$$
where $p_t^{Y(t)}$ denotes the density function of the random variable $Y(t)$ with $Y=\left(Y(t)\right)_{t \geq 0}$ is the process solution of (\ref{OU}). It is well known that $Y(t)=ye^{-\zeta t}+\int_0^te^{-\zeta(t-s)}dW(s)$ so that 
$Y$ is a Gaussian process whose mean and covariance functions are  $m(t)=ye^{-\zeta t}$ and $r(t,t+h)=e^{-\zeta h}\frac{(1-e^{-2\zeta t})}{2\zeta}$ respectively. Thus the density of the random variable $Y(t)$ is
$$p^{Y(t)}(z)=\sqrt{\frac\zeta\pi}\frac{e^{-\frac{\zeta(z-ye^{-\zeta t})^2}{(1-e^{-2\zeta t})}}}{\sqrt{1-e^{-2\zeta t}}} \, \cdot$$

As previously, we use a regression to link the centered Gaussian random variable 
$\displaystyle Z(t):=\int_0^tY(s)ds-\frac{y}{\zeta}(1-e^{-\zeta t})=\int_0^t\int_0^se^{-\zeta(s-u)}dW(u)ds=\frac1\zeta\int_0^t(1-e^{-\zeta(t-u)})dW(u)$
whose variance is
$\sigma^2_Z(t)=\frac1{\zeta^2}\int_0^t(1-e^{-\zeta(t-u)})^2du$
with the centered Gaussian random variable $\displaystyle M(t):=Y(t)-ye^{-\zeta t}$. 
With these notation, the regression model writes
$$Z(t)=\omega(t) M(t)+\xi(t), \quad \xi(t)\perp M(t)$$ with $$\omega(t)=\frac{\E[Z(t) M(t)]}{\E[M^2(t)]}=\frac1{\zeta} \frac{[1-e^{-\zeta t}]}{1+e^{-\zeta t}}=\frac1{\zeta}\tanh(\frac\zeta2t),$$ as
$\displaystyle \E[Z(t) M(t)]=\E[\frac1\zeta\int_0^t(1-e^{-\zeta(t-s_1)})dW(s_1)\int_0^te^{-\zeta(t-s_2)}dW(s_2)]=\frac1{2\zeta^2}(1-e^{-\zeta t})^2
$ and
$\displaystyle \E[M^2(t)]=\frac{(1-e^{-2\zeta t})}{2\zeta}$.
We also compute
$\displaystyle \sigma^2_{\xi(t)}:=\E[\xi^2(t)]=\E[Z^2(t)]-\omega^2(t)\E[M^2(t)]=\sigma^2_{Z(t)}-\frac1{2\zeta^3}\frac{(1-e^{-\zeta t})^3}{(1+e^{-\zeta t})}$.
Then we get\\
$\displaystyle \E^y[e^{\alpha\int_0^tY(s)ds}|Y(t)=z]=e^{\alpha\frac{y}{\zeta}(1-e^{-\zeta t})}\E[e^{\alpha Z(t)}|M(t)=z-ye^{-\zeta t}]$\\\\
$\displaystyle  =  e^{\alpha\frac{y}{\zeta}(1-e^{-\zeta t})}e^{\frac{\alpha}{\zeta}(z-ye^{-\zeta t})\tanh(\frac\zeta2t)}\E[e^{\alpha\xi(t)}]
=  e^{\alpha\frac{y}{\zeta}(1-e^{-\zeta t})}e^{-\frac{\alpha}{\zeta}(z-ye^{-\zeta t})\tanh(\frac\zeta2t)}e^{\frac12\alpha^2\sigma^2_{\xi(t)}}$ which leads
to Formula \eqref{OUFS} for $p^V_{OU}(t,y,z)$.
\end{proof}

\subsection{Application to the computation of the solution to the Ornstein-Uhlenbeck Kolmogorov hypoelliptic equation}

In this section, we consider the following partial differential equation
$$\left\{
\begin{array}{l}\label{HEeq}
  \nonumber \displaystyle \frac{\partial u}{\partial t}=\frac12\frac{\partial^2u}{\partial^2y}-\zeta y\frac{\partial u}{\partial y}-y\frac{\partial u}{\partial x}\\\\
     \nonumber  u(0,x,y)=f(x,y).
\end{array}\right.$$
This equation is a particular case of \eqref{Kolm} with $d=1$, $\tilde d=1$, $V \equiv 0$.
We aim at computing its fundamental solution, denoted by $p_{OU}^0$.

By taking the Fourier transform with respect to the variable $x$ we obtain the parabolic equation
$$\left\{
\begin{array}{l}
\displaystyle \frac{\partial \widehat{p_{OU}^0}}{\partial t}=\frac12\frac{\partial^2\widehat{p_{OU}^0}}{\partial^2y}-\zeta y\frac{\partial \widehat{p_{OU}^0}}{\partial y}-i\gamma y \widehat{p_{OU}^0}\\\\
\displaystyle \widehat{p_{OU}^0}(0,\gamma,y)=e^{-ix' \gamma} \delta_{y'}(y).
\end{array}
\right.
$$

Using Proposition \ref{prop:OU} we get
\begin{eqnarray*}\widehat{p_{OU}^0}(t,\gamma,y)&=&e^{-i\gamma x'} e^{-i\gamma\frac{y}{\zeta}(1-e^{-\zeta t})}e^{\frac{i\gamma}{\zeta}[(y'-ye^{-\zeta t})\tanh(\frac\zeta2t)]}e^{-\frac12\gamma^2\sigma^2_{\xi(t)}}\sqrt{\frac\zeta\pi}\frac{e^{-\frac{\zeta(y'-ye^{-\zeta t})^2}{(1-e^{-2\zeta t})}}}{\sqrt{1-e^{-2\zeta t}}}\\
&=&e^{-i\gamma \big[x'+\frac y{\zeta}(1-e^{-\zeta t})-\frac{1}{\zeta}(y'-ye^{-\zeta t})\tanh(\frac\zeta2t)\big]}e^{-\frac12\gamma^2\sigma^2_{\xi}(t)}\sqrt{\frac\zeta\pi}\frac{e^{-\frac{\zeta(y'-ye^{-\zeta t})^2}{(1-e^{-2\zeta t})}}}{\sqrt{1-e^{-2\zeta t}}}\cdot
\end{eqnarray*}
Then, applying the inverse Fourier transform we get
$$
p_{OU}^0(t,x,y,x',y')=\frac{\sqrt{6}\sqrt{\zeta}}{\pi t^{3/2}}\;e^{-\frac{(x-x'-\frac y{\zeta}(1-e^{-\zeta t})+\frac{1}{\zeta}(y'-ye^{-\zeta t})\tanh(\frac\zeta2t)])^2}{\frac{t^3}{6}}}\; \frac{e^{-\frac{\zeta(y'-ye^{-\zeta t})^2}{(1-e^{-2\zeta t})}}}{\sqrt{1-e^{-2\zeta t}}}\cdot
$$

\section{Small time approximation}\label{smallt}
Let us come back to the study of the Kolmogorov hypoelliptic equation introduced in \eqref{Kolm-back}. In Section \ref{KHEsection}, we focused on the study of such equations, in the particular case where $\tilde d =1$, $d \geq 1$, $b(\cdot) \equiv 0$ and $c(\cdot)$ is linear or quadratic. The aim of the present section is to study the more general case where  $\tilde d =1$, $d \geq 1$, $b(\cdot) \equiv 0$ but $c(\cdot)$ is not necessarily linear nor quadratic. That is we consider the solution $p^c(t,x,y,x',y')$ to the KHE
\begin{equation}\label{Kolm3}
  \left\{\begin{array}{l}
  \partial_tp(t,x,y,x',y') =  \frac12\Delta_yp(t,x,y,x',y')+c(y)\partial_xp(t,x,y,x',y') , \\ \qquad \qquad \qquad \qquad \qquad \qquad \qquad  \qquad \qquad (t,x,y)\in\R^{+}\times\R\times\R^d\\
  p(0,x,y,x',y') =   \delta_{x'}(x)\otimes\delta_{y'}(y),\quad (x,y)\in\R\times\R^d,
\end{array}\right.
\end{equation}
with $c(\cdot)$ not necessarily linear nor quadratic.
{\color{black}Then our main result, stated in Theorem \ref{thm:approx}, is an approximation result in small time of the semigroup $(P^c_t)$ defined as $P_t^cf(x,y) = \int_{\mathbb{R}\times \mathbb{R}^d} f(x',y')p^c(t,x,y,x',y')dx'dy'$. Note that this semigroup is associated to the SDE:
\begin{equation*}
\begin{cases}
dX(t)  =  c(Y(t))dt \\
dY(t)  =  dW(t)
\end{cases} 
\end{equation*}
with $W$ a $d$-dimensional standard Brownian motion. It corresponds to Equation \eqref{hypo1} mentioned in the introduction, with $b\equiv 0$.} 
{\color{black}The proof of Theorem \ref{thm:approx} relies on a conjecture giving an expansion in small time of $p^c$ solution of \eqref{Kolm3}. The computations leading to this conjecture are presented in the four steps below.}\\

{\bf Step 1: Fourier transform. }We first take the Fourier transform of Equation \eqref{Kolm3} with respect to the variable $x$, and get for any $\gamma\in\R$:
\begin{equation}\label{Kolm4}
  \left\{\begin{array}{l}
\displaystyle \partial_t\hat p(t,\gamma,y,x',y')=\frac12\Delta_y\hat p(t,\gamma,y,x',y')+i\gamma c(y) \hat p(t,\gamma,y,x',y') \\
 \hat p(0,\gamma,y,x',y')=e^{-ix'\gamma}\delta_{y'}(y).
\end{array}\right.
\end{equation}
The solution of Eq. \eqref{Kolm4}, given in (\ref{fundamental}), is recalled hereafter:
 $$\widehat{p^c}(t,\gamma,y,x',y')=e^{-i\gamma x'}\E[e^{i\gamma\int_0^tc(y+W(s))ds}|W(t)=y'-y]p_t(y'-y).$$
 Once more, we introduce a regression model:
 $$W(s)=\frac{s}{t}W(t)+\sqrt{t}\mathbf b\left(\frac{s}{t}\right)
,\, 0 \le s\le t,$$
where the equality is an equality in probability distribution and where $\mathbf b(\cdot)$ is a multidimensional Brownian bridge $\mathbf b(s)=(b_1(s),\ldots, b_d(s))$, having in each coordinate  independent Brownian bridges that are also independent from $W(t)$.
We get (see also \cite[p. 45]{Si:Si}):
$$\displaystyle \widehat{p^c}(t,\gamma,y,x',y')=e^{-i\gamma x'}p_t(y'-y)\E[e^{it\gamma\int_0^1c(y+s(y'-y)+\sqrt t\mathbf b(s))ds}]$$
 
 \medskip
 
 We assume in the following that $c$ is three times continuously  differentiable and the derivatives satisfy $||c'(y)||\vee ||c''(y)|| \vee ||c'''(y)||\le  C(1+||y||^2)$ for $k=0,1,2,3$ and $ C$ a generic constant. 
 \begin{Remark}Note that it is possible relax the above assumption to $||c'(y)||\vee ||c''(y)|| \vee ||c'''(y)|| \le P(||y||)$ with $P$ a polynomial with positive coefficients, leading to similar results but more complex expressions. 
 \end{Remark} 
 
 \vspace{0.2cm}
 {\bf Step 2: a first Taylor expansion.} The assumption that $c$ is two times differentiable allows to apply  a Taylor's expansion, leading to:
\begin{equation*}
\int_0^1c(y+s(y'-y)+\sqrt t\mathbf b(s))ds
 =c(y)+\frac12<c'(y),y'-y>+\sqrt t\int_0^1<c'(y),\mathbf b(s)>ds+
F(y,y',\sqrt t\mathbf b(\cdot))\end{equation*}
where we have denoted by $\displaystyle F(y,y',\sqrt t\mathbf b(\cdot))$ the integral form of the remainder term, that is\\
 $$\int_0^1\int_0^1 (s(y'-y)+\sqrt t\mathbf b(s)))^Thc''(y+(1-h)[s(y'-y)+\sqrt t\mathbf b(s)])(s(y'-y)+\sqrt t\mathbf b(s)))dhds.$$
 {\color{black}Note that for $c$ affine, this remainder term is null. Therefore we focus in the following on the case where $c$ is non affine.}
 We then write:
$$ \E[e^{it\gamma\int_0^1c(y+s(y'-y)+\sqrt t\mathbf b(s))ds}]=e^{it\gamma( c(y)+\frac{1}2<c'(y),y'-y>)}\E[e^{it^{\frac32}\gamma<c'(y),\int_0^1\mathbf b(s )ds>}e^{it\gamma F(y,y',\sqrt t\mathbf b(\cdot))}]\cdot$$

The variable  $<c'(y),\int_0^1\mathbf b(s )ds>$ is a mean zero Gaussian random variable with variance equal to\\ $\E[(\int_0^1b_1(s)ds)^2]||c'(y)||^2=\frac1{12}||c'(y)||^2$, leading to:
\begin{equation} 
\label{eq:log-norm-bis}
\E[e^{it^{\frac32}\gamma <c'(y),\int_0^1\mathbf b(s )ds>}]=e^{-\frac{\frac1{12}t^3\gamma^2||c'(y)||^2}{2}}.
\end{equation}
{\color{black}Let us define $F_{y,y',\mathbf v}: \ell \mapsto F(y,y',\ell \mathbf v)$.
 For future use, we need to bound the function $F_{y,y',\mathbf v}$ and its derivatives. We first write: 
 $$F_{y,y',\mathbf v}(\ell):=\sum_{1\le i,j\le d}\int_0^1\int_0^1L_{ij}(y,y',h,s,\ell \mathbf v)ds dh,$$ where 
 $$L_{ij}(y,y',h,s,\ell \mathbf v):=(s(y-y')+\ell \mathbf v)_ihc''_{ij}(y+(1-h)[s(y-y')+\ell \mathbf v])(s(y-y')+\ell \mathbf v)_j,$$ with the $c''_{ij}$'s being the coefficients of the matrix $c''$.} Below $\mathbf C$ denotes a generic constant  which may vary from one line to another.
 Using the hypothesis satisfied by the derivatives of $c(\cdot)$, we get the following inequality for {\color{black}$0 \leq s, h \leq 1$:}
$$\begin{array}{rcl}
|L_{ij}(y,y',h,s,\ell \mathbf v)| & \le &  ||s(y-y')+\ell \mathbf v||^2{\color{black}\sup_{ij}|c''_{ij}(y+(1-h)[s(y-y')+\ell  \mathbf v])|}\\
& \le &  \mathbf C(||y-y'||^2+||\ell \mathbf v||^2)(1+||y+(1-h)s(y-y')+\ell  \mathbf v]||^2)\\
& \le & \mathbf C(||y-y'||^2+\ell^2|| \mathbf v||^2)(1+2||y+(1-h)[s(y-y')||^2+2\ell^2||\mathbf v||^2)\\
& \le & \mathbf C(C_1(y,y')+C_2(y,y') \ell^2 ||\mathbf v||^2+C_3(y,y')\ell^4||\mathbf v||^4),
\end{array}$$
where the  $C_i(y,y')$ are up to order 4 polynomials depending on $y,y'$. This yields
\begin{equation*}
{\color{black}|F_{y,y',\mathbf v}(\ell)|}\le\mathbf C(C_1(y,y')+C_2(y,y') \ell^2 ||\mathbf v||^2+C_3(y,y')\ell^4||\mathbf v||^4)
\end{equation*}
and finally
\begin{equation}\label{cota2}|F(y,y',\sqrt t\mathbf b(\cdot))|^2{\color{black}=|F_{y,y',\mathbf b(\cdot)}(\sqrt t )|^2}\le \mathbf C(C^2_1(y,y')+C^2_2(y,y') t^2 ||\mathbf b ||^4_{\infty}+C^2_3(y,y')t^4||\mathbf b||^8_{\infty}).
\end{equation}

{\color{black}We now bound the derivative of $F_{y,y',\mathbf v}$.
This derivative writes
$$F'_{y,y',\mathbf v}(\ell)=\sum_{1\le i,j\le d}\int_0^1\int_0^1\partial_{\ell} L_{ij}(y,y',h,s,\ell \mathbf v)ds dh,$$}
{\color{black}with $\partial_{\ell} L_{ij}(y,y',h,s,\ell \mathbf v)$ denoting the first-order derivative of $\ell \mapsto L_{ij}(y,y',h,s,\ell \mathbf v)$. Thus, to bound $F'_{y,y',\mathbf v}(\ell)$, we first bound $\partial_{\ell} L_{ij}(y,y',h,s,\ell \mathbf v)$:}
$$\begin{array}{lcl}
|\partial_{\ell} L_{ij}(y,y',h,s,\ell \mathbf v)| & \leq & |v_ihc''_{ij}(y+(1-h)[s(y-y')+\ell  \mathbf v])(s(y-y')+\ell \mathbf v)_j|\\\\
&  & \, +|v_jhc''_{ij}(y+(1-h)[s(y-y')+\ell \mathbf v])(s(y-y')+\ell \mathbf v)_i|\\\\
& & \,  +|(s(y-y')+\ell \mathbf v)_i(s(y-y')+\ell \mathbf v)_j\sum_{k=1}^d hc'''_{ijk}(y+(1-h)[s(y-y')+\ell \mathbf v])v_k|.\end{array}$$
Using again the bound for $||c''(\cdot)||$ we obtain the following upper bound for the two first terms:
$$  C_4(y,y')\left(||\mathbf v||+\ell ||\mathbf v||^2+\ell^2 ||\mathbf v||^3+\ell^3 ||\mathbf v||^4\right)$$
with $C_4(y,y')$ a polynomial of order 4.
For bounding the third term we use the  bound on $||c'''(\cdot)||$. It finally leads to:
$$\begin{array}{ll}
|(s(y-y')+\ell \mathbf v)_i(s(y-y')+\ell \mathbf v)_j\sum_{k=1}^d hc'''_{ijk}(y+(1-h)[s(y-y')+\ell \mathbf v])v_k|&\\ 
\le C_5(y,y')\left(1 + \ell^2||\mathbf v||^2+\ell^4||\mathbf v||^4\right)&
\end{array}$$
with $C_5(y,y')$ a polynomial of order 4.
Summing up we have
\begin{equation}
\label{eq:cont-dF}
|{\color{black}F'_{y,y',\mathbf v}(\ell)}|\le C(y,y') \left(1 + ||\mathbf v||+\ell ||\mathbf v||^2+\ell^2||\mathbf v||^2+\ell^2 ||\mathbf v||^3+\ell^3 ||\mathbf v||^4+\ell^4||\mathbf v||^4\right),
\end{equation}
with $C(y,y')$ a polynomial in $y,y'$, of order 4.

\vspace{0.4cm}

{\bf Step 3: a second Taylor expansion.}
Now we consider the following expansion:

$$\begin{array}{lcl}
\frac{1}{t} \left(\widehat{ p^c}(t,\gamma,y, x',y')-e^{-i\gamma x'}e^{it\gamma( c(y)+\frac{1}2<c'(y),y'-y>)}e^{-\frac12(\frac1{12}t^3\gamma^2||c'(y)||^2)}\right)p_t(y-y')&&\\\\
\;  =e^{-i\gamma x'}e^{it\gamma( c(y)+\frac{1}2<c'(y),y'-y>)}\E[e^{it^{\frac32}\gamma<c'(y),\int_0^1\mathbf b(s )ds>}\frac{e^{it\gamma F(y,y',\sqrt t\mathbf b(\cdot))}-1}{t}] p_t(y-y')&&\\\\
\;=e^{-i\gamma x'}e^{it\gamma( c(y)+\frac{1}2<c'(y),y'-y>)}\E[e^{it^{\frac32}\gamma <c'(y),\int_0^1\mathbf b(s )ds>}(\frac{e^{i\gamma t F(t,z,y,\sqrt t\mathbf b(\cdot))}-1-it\gamma F(y,y',\sqrt t\mathbf b(\cdot))}t)]p_t(y'-y)\\\\
\;+e^{-i\gamma x'}e^{it\gamma( c(y)+\frac{1}2<c'(y),y'-y>)}\E[e^{it^{\frac32}\gamma <c'(y),\int_0^1\mathbf b(s )ds>}(i\gamma F(y,y',\sqrt t\mathbf b(\cdot))]p_t(y'-y)=I_1+I_2.\\\\
\end{array}$$
Let us study each term in the above sum separately. Applying the following inequality, $\forall \, t>0 , \, \forall \, x\in\R,$ 
$|e^{itx}-1-itx|\le\frac{t^2x^2}2$, we obtain:
\begin{equation}\label{ref:eqnew}
|\frac{e^{it\gamma F(y,y',\sqrt t\mathbf b(\cdot))}-1-it\gamma F(y,y',\sqrt t\mathbf b(\cdot))}{t}|\le \frac{t\gamma^2F^2(y,y',\sqrt t\mathbf b(\cdot))}2\cdot \end{equation}
Then we bound $I_1$ using \eqref{cota2}, \eqref{ref:eqnew} and the integrability of $||\mathbf b||^\alpha_{\infty}$ for $\alpha>0$:
\begin{equation}\label{cota8}|I_1|\le t\gamma^2\mathbf C\E[(C^2_1(y,y')+C^2_2(y,y') t^2 ||\mathbf b ||^4_{\infty}+C_3^2(y,y')t^4||\mathbf b||^8_{\infty})].\end{equation}
{\color{black}For bounding $I_2$, we use a Taylor expansion of $F_{y,y', \mathbf b(\cdot)}$:
$$F_{y,y', \mathbf b(\cdot)}(\sqrt t) =F_{y,y', \mathbf b(\cdot)}(0)+\sqrt t  \int_0^1{\color{black}F'_{y,y', \mathbf b(\cdot)}(z\sqrt{t})} \mathbf dz.$$}
Then, {\color{black}using \eqref{eq:cont-dF}}, we get:
\begin{equation}\label{cota9}
 \begin{array}{ll}
 |\sqrt t \int_0^1{\color{black}F'_{y,y', \mathbf b(\cdot)}(z\sqrt{t})} \mathbf dz|&\\
 \le \sqrt t C(y,y')\left(1+||\mathbf b||_\infty+\sqrt t ||\mathbf b||_\infty^2+t ||\mathbf b||_\infty^2+t ||\mathbf b||_\infty^3+t^{\frac 32} ||\mathbf b||_\infty^4+t^2 ||\mathbf b||_\infty^4\right),&
 \end{array}
    \end{equation}
    with $C(y,y')$ a polynomial in $y$, $y'$ up to order 4.
Using again the integrability of $||\mathbf b||^\alpha_\infty$, we finally get:
\begin{equation}
\label{eq:Eeit3}
\E[e^{it^{\frac32}\gamma <c'(y),\int_0^1\mathbf b(s )ds>}(i\gamma F(y,y',\sqrt t\mathbf b(\cdot))]
=i\gamma F(y,y',0)\E[e^{it^{\frac32}\gamma <c'(y),\int_0^1\mathbf b(s )ds>}]+\gamma \mathcal{O}(\sqrt t).\end{equation}



Finally, gathering all the bounds in Step 3 we obtain:
	\begin{multline}\label{eq:new}
	\frac{1}{t} \left(\widehat{ p^c}(t,\gamma,y, x',y')-e^{-i\gamma x'}e^{it\gamma( c(y)+\frac{1}2<c'(y),y'-y>)}e^{-\frac12(\frac1{12}t^3\gamma^2||c'(y)||^2)}\right)p_t(y-y')\\\\
	\; = e^{-i\gamma x'}e^{it\gamma( c(y)+\frac{1}2<c'(y),y'-y>)}\left[e^{-\frac12(\frac1{12}t^3\gamma^2||c'(y)||^2)}i\gamma H(y,y')+\gamma \mathcal{O}\left(\sqrt{t}\right)+\gamma^2 \mathcal{O}\left(t\right)\right]p_t(y-y'),\end{multline}
\begin{equation}
\label{eq:def-H}
\textup{ with }H(y,y')=F(y,y',0)=\int_0^1\int_0^1s^2(y'-y)^Thc''(y+(1-h)s(y'-y))(y'-y)dhds.
\end{equation}
\begin{Remark}
\label{rem-O}
Note that the multiplicative constant in the term $\mathcal{O}\left(\sqrt{t}\right)$ writes as a polynomial in $y,y'$, up to order $8$. Here the remainder term in $\gamma \mathcal{O}(\sqrt t)$ comes from the contribution in $\gamma \mathcal{O}(\sqrt t)$ from $I_2$ and the contribution in
$\gamma^2 \mathcal{O}(t)$ from $I_1$. We will have to consider these contributions separately in the proof of Theorem \ref{thm:approx}.
\end{Remark}
%
%


\vspace{0.2cm}

{\bf Step 4: computing the inverse Fourier transform.}
From now on we assume that we are doing our computations with $y\in\R^d$ s.t. $c'(y)\neq 0$ (see the forthcoming Remark \ref{rem:y-zero}). Let us denote 
\begin{equation}
\label{eq:q}
\displaystyle q(t,x,y,x',y')=\frac{\sqrt6e^{-6\frac{((x-x')+t[c(y)+\frac12<c'(y),y'-y>])^2}{t^3||c'(y)||^2}}}{\sqrt{\pi}||c'(y)||t^{\frac32}}\,p_t(y-y').
\end{equation}
We have
$$\displaystyle \hat q(t,\gamma,y,x',y')=e^{-i\gamma x'}e^{it\gamma( c(y)+\frac{1}2<c'(y),y'-y>)}e^{-\frac12(\frac1{12}t^3\gamma^2||c'(y)||^2)}p_t(y-y').$$
Now, from \eqref{eq:new} and using then the inverse Fourier transform we conjecture that:
	\begin{equation}\label{conconj}\begin{array}{lcl}
	\frac1t \quad (p^c(t,x,y,x',y')-q(t,x,y,x',y')) &&\\\\
	=\frac1{2\pi}\int_{\R}e^{i\gamma x}(\frac{\widehat{p^c}(t,\gamma,y,x',y')-\hat q(t,\gamma,y,x',y')}t)d\gamma &&
	\\\\
	=-\frac2{\sqrt\pi}\left[H(y,y')+\mathcal{O}\left(\sqrt{t}\right)\right]p_t(y-y')(\frac6{t^3||c'(y)||^2})^{\frac32} &&\\\\
	\qquad \qquad \times(x-x'+t( c(y)+\frac{1}2<c'(y),y'-y>))e^{-\frac{6(x-x'+t( c(y)+\frac{1}2<c'(y),y'-y>))^2}{t^3||c'(y)||^2}}&&
	\end{array}\end{equation}
	thus
	\begin{equation}\label{det}
	\frac{p^c(t,x,y,x',y')}{q(t,x,y,x',y')}
=1-\frac{12(x-x'+t( c(y)+\frac
12<c'(y),y'-y>))\left(H(y,y')+\mathcal{O}\left(\sqrt{t}\right)\right)}{t^{2}||c'(y)||^2}
	\end{equation}
	with $\mathcal{O}\left(\sqrt{t}\right)$ a polynomial in $y$, $y'$, up to order $8$ (see again Remark \ref{rem-O}). Note that this result is only a conjecture because we have no guarantee that the inverse Fourier transform of the remainder terms in \eqref{eq:new}, namely the terms $\gamma \mathcal{O}(\sqrt t)$ and $\gamma^2 \mathcal{O}(t)$, leads to the term $\mathcal{O}(\sqrt t)$ in \eqref{conconj}.


\bigskip

{\color{black}Finally, Steps 1 to 4 lead to Conjecture \ref{conj:1} below:

	\begin{conj}
		\label{conj:1}
		Let $H(y,y')$ and $q(t,x,y,x',y')$ be defined respectively by \eqref{eq:def-H} and \eqref{eq:q}. \\
		
		\noindent $\bullet$ Assume $c$ is affine non constant. Let $(x,x',y,y') \in \mathbb{R}^2\times \mathbb{R}^{2d}$. Then, $$p^c(t,x,y,x',y')=q(t,x,y,x',y').$$
  Note that in the case where $c$ is linear we recover the result of Corollary \ref{cor:linearcase}.\\
	\noindent $\bullet$ 	Assume $c$ is three times continuously differentiable and satisfies $||c'(y)||\vee ||c''(y)|| \vee ||c'''(y)||\le C \left(1+||y||^2\right)$. Let $(x,x',y,y') \in \mathbb{R}^2\times \mathbb{R}^{2d}$ be such that $c'(y)\neq 0$. Define
	\begin{equation}\label{eq:p-bar}
\begin{array}{ll}
 \overline{p}^c(t,x,y,x',y')&=q(t,x,y,x',y')\left(\frac{t^2\|c'(y)\|^2-12  \left( x-x'+t( c(y)+\frac{1}2<c'(y),y'-y>)\right)H(y,y')}{t^2\|c'(y)\|^2}\right)\\\\
 &=  \frac{\sqrt{12}}{(2\pi)^{\frac{d+1}{2}}} \frac{1}{\|c'(y)\|^3 t^{\frac{7+d}{2}}} \exp \left(-\frac{\|y-y'\|^2}{2t}-\frac{6\left(x-x'+t( c(y)+\frac{1}2<c'(y),y'-y>)\right)^2}{t^3||c'(y)||^2}\right)  \\\\
& \; \quad\quad \times \left[t^2\|c'(y)\|^2-12  \left( x-x'+t( c(y)+\frac{1}2<c'(y),y'-y>)\right)H(y,y')\right]\cdot
\end{array}
\end{equation}
Then, for $(x,x',y,y') \in \mathbb{R}^2 \times \mathbb{R}^{2d}$ such that $12  \left( x-x'+t( c(y)+\frac{1}2<c'(y),y'-y>)\right)H(y,y') \neq 0$, we deduce from \eqref{det}:
\begin{equation}
\label{eq:p}
\frac{p^c(t,x,y,x',y')}{\overline{p}^c(t,x,y,x',y')}
=1- \mathcal{O}\left(\sqrt t \right) \frac{1}{H(y,y')}
\end{equation}
with $\mathcal{O}\left(\sqrt{t}\right)$ a polynomial in $y$, $y'$, up to order $8$. 
\end{conj}}

{\color{black}
Let us define 
%
$\displaystyle \overline{P}^c_tf(x,y)=\int_{\mathbb{R}^{1+d}} \overline{p}^c(t,x,y,x',y')f(x',y')dx'dy'$.
}
We now state in Theorem \ref{thm:approx} below the main result of this section.

\begin{thm}
\label{thm:approx}
Assume $c$ is three times continuously differentiable  and that $||c'(y)||\vee ||c''(y)|| \vee ||c'''(y)||\le C \left(1+||y||^2\right)$.
Let $f\in C_b(\R\times\R^d)$ such that $f(\cdot,y')$, $\frac{\partial f}{\partial x}(\cdot,y')$, $\frac{\partial^2 f}{\partial x^2}(\cdot,y')$ and $\frac{\partial^3 f}{\partial x^3}(\cdot,y')$ are square integrable for all $y'$.  
Let $(x,y)$ s.t. $c'(y)\neq 0$ and assume that for any polynomial $C(y,y')$ up to order $8$ one has
\begin{equation}\label{bound}
 \int_{\R^d} |C(y,y')|(||f(\cdot,y')||_2+||\frac{\partial f}{\partial x}(\cdot,y')||_2+||\frac{\partial^2 f}{\partial x^2}(\cdot,y')||_2 + ||\frac{\partial^3 f}{\partial x^3}(\cdot,y')||_2)p_t(y'-y)dy' 
<\infty.\end{equation}
Then we have, for $t\downarrow 0$,
$$
\big| P^c_tf(x,y)-\overline{P}^c_tf(x,y) \big|=\mathcal{O}(t^{\frac32}).
$$
\end{thm}

\begin{proof}
Denote $\hat f(\gamma,y')=\int_\R e^{-i\gamma x}f(x,y')dx$. We first check that for any 
$y'\in\R^d$ we have $\hat f(\cdot,y')\in \mathbb L^1(\R)$. Indeed we have that
\begin{equation}
\label{eq:dom}
\begin{array}{rcl}
\int_\R|\hat f(\gamma,y')|d\gamma & \le &  (\int_\R(1+\gamma^2)|\hat f(\gamma,y')|^2d\gamma)^{\frac12}(\int_\R(1+\gamma^2)^{-1}d\gamma)^{\frac12}\\\\
& \le &  C (||f(\cdot,y')||_2+||\frac{\partial f}{\partial x}(\cdot,y')||_2)<\infty.\end{array}
\end{equation}
In the same manner we can show that $\hat g(\cdot,y')\in \mathbb L^1(\R)$ for $g(\cdot)=\frac{\partial f}{\partial x}(\cdot,y')$ and $g(\cdot)=\frac{\partial^2 f}{\partial x^2}(\cdot,y')$, using our assumption on the partial derivatives of $f$. This will be needed in the sequel to control terms one may find in integrals.

We recall that $u(t,x,y)=P^c_tf(x,y)$ solves
\begin{equation}\label{Kolm5}
  \left\{\begin{array}{l}
\displaystyle \frac{\partial u}{\partial t}=\frac12\Delta_yu+c(y)\frac{\partial u}{\partial x} \\\\
 u(0,x,y)=f(x,y).\, 
\end{array}\right.\end{equation} To solve Equation (\ref{Kolm5}), we take the Fourier transform with respect to the variable $x$, leading to, for any $\gamma$,
$$\left\{
\begin{array}{l}\label{Kolm6}
\displaystyle \frac{\partial \hat u}{\partial t}=\frac12\Delta_y\hat u+i\gamma c(y)\hat u \\\\
 \hat u(0,\gamma,y)= \hat f(\gamma,y).
\end{array}
\right.$$
The solution of this last equation is (see Section \ref{sec:FK})
$$\begin{array}{rcl}
\hat u(t,\gamma,y)& = & \E[e^{i\gamma\int_0^t c(y+W(s))ds} \hat f(\gamma,y+W(t))]\\\\
& = & \int_{\R^d}\E[e^{it\gamma\int_0^1c(y+s(y'-y)+\sqrt t\mathbf b(s))ds}]\hat f(\gamma,y')p_t(y'-y)dy'.\end{array}$$

Let us introduce the {\it frozen solution}
$$\hat u_{fr}(t,\gamma,y)=\int_{\R^d}e^{it\gamma(c(y)+\frac{1}2<c'(y),y'-y>)}e^{-\frac12(\frac1{12}t^3\gamma^2||c'(y)||^2)}\hat f(\gamma,y')p_t(y'-y)dy'.$$
We get, using \eqref{eq:new}: 
\begin{equation}\label{thm:eq}
\begin{array}{ll}
 \displaystyle \frac{\hat u(t,\gamma,y)- \hat u_{fr}(t,\gamma,y)}t&\\\\
 \displaystyle =\int_{\R^d}e^{it\gamma(c(y)+\frac{1}2<c'(y),y'-y>)}\E[e^{it^{\frac32}\gamma <c'(y),\int_0^1\mathbf b(s )ds>}\frac{(e^{i\gamma t F(y,y',\sqrt t\mathbf b(\cdot))}-1)}t]\hat f(\gamma,y')p_t(y'-y)dy'&\\\\
\displaystyle =\int_{\R^d}e^{it\gamma(c(y)+\frac{1}2<c'(y),y'-y>)}e^{-\frac12(\frac1{12}t^3\gamma^2||c'(y)||^2)}i\gamma H(y',y)\hat f(\gamma,y')p_t(y'-y)dy' & \\\\
\displaystyle \qquad   +\int_{\R^d}\big[ \gamma \mathcal{O}(\sqrt t)+\gamma^2 \mathcal{O}(t)\big]\hat f(\gamma,y')p_t(y'-y)dy'.&\end{array}\end{equation}
We now wish to take the inverse Fourier transform in \eqref{thm:eq}. To compute the inverse Fourier transform, we apply Fubini's theorem.

Let us first consider the main term in the right hand side of \eqref{thm:eq}. Using the relation $\widehat{\phi'}(\gamma)=i\gamma\hat{\phi}(\gamma)$, the control \begin{equation}
\label{eq:cont-f1chap}
    \int_\R|\widehat{\frac{\partial f}{\partial x}}(\cdot,y')|d\gamma\leq C (||\frac{\partial f}{\partial x}(\cdot,y')||_2+||\frac{\partial^2 f}{\partial x^2}(\cdot,y')||_2)
    \end{equation}
    (obtained as in \eqref{eq:dom}) and hypothesis \eqref{bound} allows to perform Fubini theorem. Then, using the relation $i\gamma\widehat{\psi}(\gamma)\widehat{\phi}(\gamma)=\widehat{\psi'*\phi}(\gamma)$ with $\widehat\phi(\gamma)=e^{it\gamma(c(y)+\frac{1}2<c'(y),y'-y>)}e^{-\frac12(\frac1{12}t^3\gamma^2||c'(y)||^2)}$ and $\widehat\psi(\gamma)=\hat f(\gamma,y')$, we get that the inverse Fourier transform of the main term is
$$\begin{array}{l}
\int_{\R^{1+d}}\frac2{\sqrt \pi}H(y',y)(\frac 6{ t^3||c'(y)||^2})^{\frac32}((x-x')+t[c(y)+\frac12<c'(y),y'-y>])\\\\
 \qquad \times e^{-\frac{6((x-x')+t[c(y)+\frac12<c'(y),y'-y>])^2}{t^3||c'(y)||^2}}f(x',y')p_t(y'-y)dx'dy'\,. \\
\end{array}$$
In the same manner the inverse Fourier transform of $\hat u_{fr}(t,\gamma,y)$ is equal to
$$
\int_{\R^{1+d}}\frac{\sqrt6e^{-6\frac{((x-x')+t[c(y)+\frac12<c'(y),y'-y>])^2}{t^3||c'(y)||^2}}}{\sqrt{\pi}||c'(y)||t^{\frac32}}\,f(x',y')p_t(y'-y)dx'dy'.
$$
We then compute the inverse Fourier transform of the remainder terms in \eqref{thm:eq}. Let us for example consider the term 
$$\int_{\R^d} \gamma \mathcal{O}(\sqrt t)\hat f(\gamma,y')p_t(y'-y)dy'=\int_{\R^d}  \frac 1 i \mathcal{O}(\sqrt t)i\gamma\hat f(\gamma,y')p_t(y'-y)dy'.$$
We recall that
the term $\mathcal{O}\left(\sqrt{t}\right)$ writes as a polynomial in $y,y'$, up to order $8$ (see Remark \ref{rem-O}). Using once more $\widehat{\phi'}(\gamma)=i\gamma\widehat{\phi}(\gamma)$, Eq. \eqref{eq:cont-f1chap} and \eqref{bound}, one may again apply Fubini's theorem and see that the inverse Fourier transform is controlled by
$$\int_{\R^d}  \frac 1 i \mathcal{O}(\sqrt t)C (||\frac{\partial f}{\partial x}(\cdot,y')||_2+||\frac{\partial^2 f}{\partial x^2}(\cdot,y')||_2) p_t(y'-y)dy'.$$
Using again \eqref{bound} allows to perform dominated convergence and to see that this term behaves as $\mathcal \mathcal{O}(\sqrt{t})$. We may proceed in the same manner for the other part of the remainder term and see that it behaves as $\mathcal \mathcal{O}(t)$.

To sum up we get
$$\begin{array}{ll}
u(t,x,y)=\int_{\R^{1+d}}\frac{\sqrt6e^{-6\frac{((x-x')+t[c(y)+\frac12<c'(y),y'-y>])^2}{t^3||c'(y)||^2}}}{\sqrt{\pi}||c'(y)||t^{\frac32}}\,f(x',y')p_t(y'-y)dx'dy'
&\\\\
\qquad \qquad \qquad-t\int_{\R^{1+d}}\frac2{\sqrt \pi}H(y',y)(\frac 6{ t^3||c'(y)||^2})^{\frac32}((x-x')+t[c(y)+\frac12<c'(y),y'-y>])\\\\
\qquad \qquad \qquad \qquad \qquad e^{-\frac{6((x-x')+t[c(y)+\frac12<c'(y),y'-y>])^2}{t^3||c'(y)||^2}}f(x',y')p_t(y'-y)dx'dy'+\mathcal{O}(t^{\frac32})\\
\hspace{1.5cm}=\int_{\R^{1+d}}\bar{p}^c(t,x,y,x',y')dx'dy'+\mathcal{O}(t^{\frac32}).&
\end{array}$$
\end{proof}

\begin{Remark}
\label{rem:y-zero}
If we consider a point $(x,y)$ with $c'(y)=0$ the approximated kernel $\bar{p}^c$ cannot be defined by \eqref{eq:p-bar}. It is however possible to perform in this case an ad hoc Taylor expansion.
\end{Remark}

\noindent {\bf Example 1:}
As an example, let us study the result of the conjecture \eqref{conconj} in the special case $c(y)=-\beta||\Omega^{\frac12}y||^2$ for $\beta=i$ or $\beta=-1$. Then
$c'(y)=-2\beta\Omega y$ and $c''(y)=-2\beta\Omega$. Furthermore, $H(y',y)=-\frac13 \beta||\Omega^{\frac12}(y-y')||^2$.
Then

$$\begin{array}{l}p^{-\beta||\Omega^{\frac12}\cdot||^2}(t,x,y,x',y')= \frac{\sqrt{12}}{(2\pi)^{\frac{d+1}{2}}} \frac{1}{2^3\|\Omega y\|^3 t^{\frac{7+d}{2}}}\,\exp{- \big(\frac{\|y-y'\|^2}{2t}+\frac{6\left(x-x'-\beta<\Omega y,y'-y>)\right)^2}{4t^3||\Omega y||^2}\big)} \times\\
\displaystyle \quad\big[4t^2\|\Omega y\|^2-12\big( x-x'-t(\beta( \|\Omega^{\frac12}y\|^{2}+<\Omega y,y'-y>)\big)\big(-\frac13\beta\|\Omega^{\frac12}(y'-y)\|^2+\mathcal{O}\big(\sqrt{t}\big)\big)\big].
\end{array}$$
{\color{black}
\noindent {\bf Example 2:} Consider the system
\begin{equation}\label{hypo2}
\begin{cases}
dZ_1(t)  =  Z_2(t)dt \\
Z_2(t)  = \varphi(W(t)) 
\end{cases} 
\end{equation}
with the function $\varphi$ satisfying the same assumptions as $c$ in 
Theorem \ref{thm:approx}, and being invertible with smooth $\varphi^{-1}$.

Note that using Itô formula, \eqref{hypo2} can be rewritten as
\begin{equation*}
\begin{cases}
dZ_1(t)  =  Z_2(t)dt \\
dZ_2(t)  = \varphi'\circ\varphi^{-1}(Z_2(t))dW(t)+\frac12\varphi''\circ\varphi^{-1}(Z_2(t))dt.
\end{cases} 
\end{equation*}
Thus if $\varphi'$, $\varphi''$ are bounded and $|\varphi'|^2$ is uniformly strictly elliptic one sees that $(Z_1,Z_2)$ is an hypoelliptic diffusion.

Consider the function $\Phi(x,y)=(x,\varphi(y))$ and set
$(X,Y)=\Phi^{-1}(Z_1,Z_2)$. It is clear that $(X,Y)$ solves the~SDE
\begin{equation*}
\begin{cases}
dX(t)  =  \varphi(Y_t)dt \\
dY(t)  =  dW(t).
\end{cases} 
\end{equation*}
Denoting $p^\varphi(t,x,y,x',y')$ the transition function of 
$(X,Y)$ and using the change of variable $(Z_1,Z_2)=\Phi(X,Y)$, we get that the transition function of $(Z_1,Z_2)$ is given by
$$
p^\varphi(t,z_1,\varphi^{-1}(z_2),z_1',\varphi^{-1}(z_2'))\times
|(\varphi^{-1})'(z_2')|
$$
Thus an approximation of this transition function is given by
$$
\bar{p}^\varphi(t,z_1,\varphi^{-1}(z_2),z_1',\varphi^{-1}(z_2'))\times
|(\varphi^{-1})'(z_2')|
$$
with $\bar{p}^\varphi$ given by formula \eqref{eq:p-bar}.

}

\section{Numerical experiments}\label{sec:num}
The aim of this section is to experiment on a simple example the practical efficiency of the approximation stated in Theorem \ref{thm:approx}.
We consider the hypoelliptic PDE
\begin{equation}
\label{eq:u-quad-test}
    \left\{\begin{array}{lcl}
\displaystyle\frac{\partial u}{\partial t}&=&\displaystyle\frac12\frac{\partial^2 u}{\partial^2 y}+c(y)\frac{\partial u}{\partial x} , \quad
(t,x,y)\in\R_+^*\times\R\times\R\\
u(0,x,y)&=&f(x,y), \quad \forall(x,y)\in\R\times\R .\\
\end{array}\right.
\end{equation}
We recall that the 
 fundamental solution to \eqref{eq:u-quad-test} and the associated semigroup are respectively denoted by~$p^c(t,x,y,x',y')$ and $(P^c_t)$. We recall that 
 \begin{equation}
 \label{eq:feyn-num}
 u(t,x,y)=P^c_tf(x,y)=\int_{\R^2}p^c(t,x,y,x',y')f(x',y')dx'dy'=\E^{x,y}[f(X_t,Y_t)]
 \end{equation}
 with $(X,Y)$ the hypoelliptic diffusion solving Eq. \eqref{hypo1} with $b\equiv 0$.

We aim at checking experimentally the validity of our approximation
 $\overline{P}^c_h$ of $P^c_h$, for small time $h$ (Theorem \ref{thm:approx}).
 
 More precisely the idea is the following. Take $0<T<\infty$, and $N\in\mathbb{N}^*$. Then the solution
 $u(T,x,y)=P^c_Tf(x,y)=P^c_{\frac T N}\circ\ldots P^c_{\frac T N}f(x,y)$
of \eqref{eq:u-quad-test} should be approached by 
\begin{equation}
\label{eq:approN}
    \overline{P}^c_{\frac T N}\circ\ldots \overline{P}^c_{\frac T N}f(x,y).
\end{equation}
Note that we could simply try to approach $P^c_Tf(x,y)$ by $\overline{P}^c_Tf(x,y)$ (i.e. take $N=1$ in~\eqref{eq:approN}).
But this would be valid only for small $T$. For large $T$ one can expect that the approximation would be better if we iterate $N$ times the approximated semigroup $\overline{P}^c_{\frac T N}$ (as~$T/N$ is small for large $N$). Also one wishes to check that the errors will not accumulate by iterating the approximated semigroup $\overline{P}^c_{\frac T N}$.

In our simulation the integral that leads to a quantity of type
\begin{equation}
\label{eq:int}
\overline{P}^c_{\frac T N}\phi(x,y)=\int_{\R^2}\overline{p}^c(\frac T N,x,y,x',y')\phi(x',y')dx'dy'
\end{equation}
is approximated by some quadrature method (see details in Example 1; here $\phi$ maybe the initial condition $f$ or some previous approximation of $\overline{P}^c_{\frac{kT}{N}}f$, $1\leq k\leq N-1$). 

\vspace{0.2cm}
We need benchmarks to which we can compare our small time approximation.

First we use a finite elements methods (space discretization) together with a
Crank-Nicolson scheme (time discretization) to solve the PDE \eqref{eq:u-quad-test}. This method will be referred to as Finite Elements (FE). We use a high discretization order to ensure the FE is at convergence.

Second we use an Euler type scheme with time step $T/n$ for the simulation of paths of $(X,Y)$ solution of \eqref{hypo1} with $b\equiv 0$, starting from $(x,y)$. We draw a large number $M$ of independent realizations of $(X_T,Y_T)$ and compute a Monte Carlo average to approach $u(t,x,y)$ (through the Feynman-Kac representation \eqref{eq:feyn-num}). This method will be referred to as Monte Carlo (MC). We use a high discretization order $n$ and a large number of samples $M$ to ensure the MC is at convergence.

\vspace{0.3cm}

\noindent
{\bf Example 1: toy example.} Here we take $c(y)=\frac 1 4 (-\frac{y^2}{2}+6y)$ and
\begin{equation}
\label{eq:CI-num}
f(x,y)=\frac{1}{2\pi\sigma_c^2}\exp(-\frac{x^2+y^2}{2\sigma_c^2})
\end{equation}
with $\sigma_c^2=0.2$. The time horizon is $T=2.5$.

Note that on our example we have $c'(y)=0$ for $y=6$. But when computing the integral 
\eqref{eq:int} (on even solving the PDE by FE) we use a bounded domain
$K=(-14,14)\times(-5,5)$ which does not intersect with the line $y=6$. Indeed with the initial condition~\eqref{eq:CI-num} the mass remains concentrated at time $T=2.5$ in $K$ and is near to zero at the boundary~$\partial K$. 

In other words we first approach
 $P^c_{T/N}\phi(x,y)=\int_{\R^2}p^c(\frac T N,x,y,x',y')\phi(x',y')dx'dy',$ by\\ $\int_{K}p^c(\frac T N,x,y,x',y')\phi(x',y')dx'dy'$ and consider that on $K$ the kernel $p^c$ is correctly approached by $\overline{p}^c$ (Conjecture \ref{conj:1} or Theorem~\ref{thm:approx}). Then the approximation by quadrature of
$\int_{K}\overline{p}^c(\frac T N,x,y,x',y')\phi(x',y')dx'dy'$ should be a correct approximation of $P^c_{T/N}\phi(x,y)$, allowing the computation of $u(T,x,y)$.

\vspace{0.2cm}
Table \ref{tab:1} shows the relative $L^p$-error, for $p=1,2,\infty$, between the FE reference solution and 
\eqref{eq:approN} computed with $N=1$ and $N=5$ (in the latter case we then have
$T/N=0.5$). Note that given the discretization grid ${(x_i,y_i)}_{i=1}^{N_x\times N_y}$of the domain $K$ the relative $L^p$-distance between a reference function $f$ and an approximating function $g$ is defined by
$$
\frac{\Big(\sum_{i=1}^{N_x\times N_y}|g(x_i,y_i)-f(x_i,y_i)|^p\Big)^{1/p}}{\Big(\sum_{i=1}^{N_x\times N_y}|f(x_i,y_i)|^p\Big)^{1/p}},
$$
for $p=1,2$ and by
$$
\frac{\max_{i=1}^{N_x\times N_y}|g(x_i,y_i)-f(x_i,y_i)|}{\max_{i=1}^{N_x\times N_y}|f(x_i,y_i)|},
$$
for $p=\infty$.

\vspace{0.5cm}
\begin{table}[h!]
 \begin{center}
\begin{tabular}{cccc}
\hline
 & $N=1$ &  $N=5$    \\
 \hline
$L^1$-error & 0.1257883    &  0.01681095  \\
\hline
$L^2$-error &  0.3115323   &  0.03828152  \\
\hline
$L^\infty$-error &  0.1732925   &  0.01735233  \\
\hline
\end{tabular}
\caption{Relative error between  the FE reference solution and the iterated semigroup \eqref{eq:approN} with $N=1$ and $N=5$}
\label{tab:1} 
\end{center}
\end{table}
As expected the result is much better with $N=5$. To illustrate this we plot several graphs. On Figure \ref{fig1} we present a 3D plot of the reference solution computed by FE.
The solution computed by \eqref{eq:approN} with $N=5$ gives a plot that is very similar to FE, so instead we plot the solution computed with $N=1$ on Figure \ref{fig2}.

\begin{figure}
\begin{center}
\includegraphics[angle=270,width=9cm]{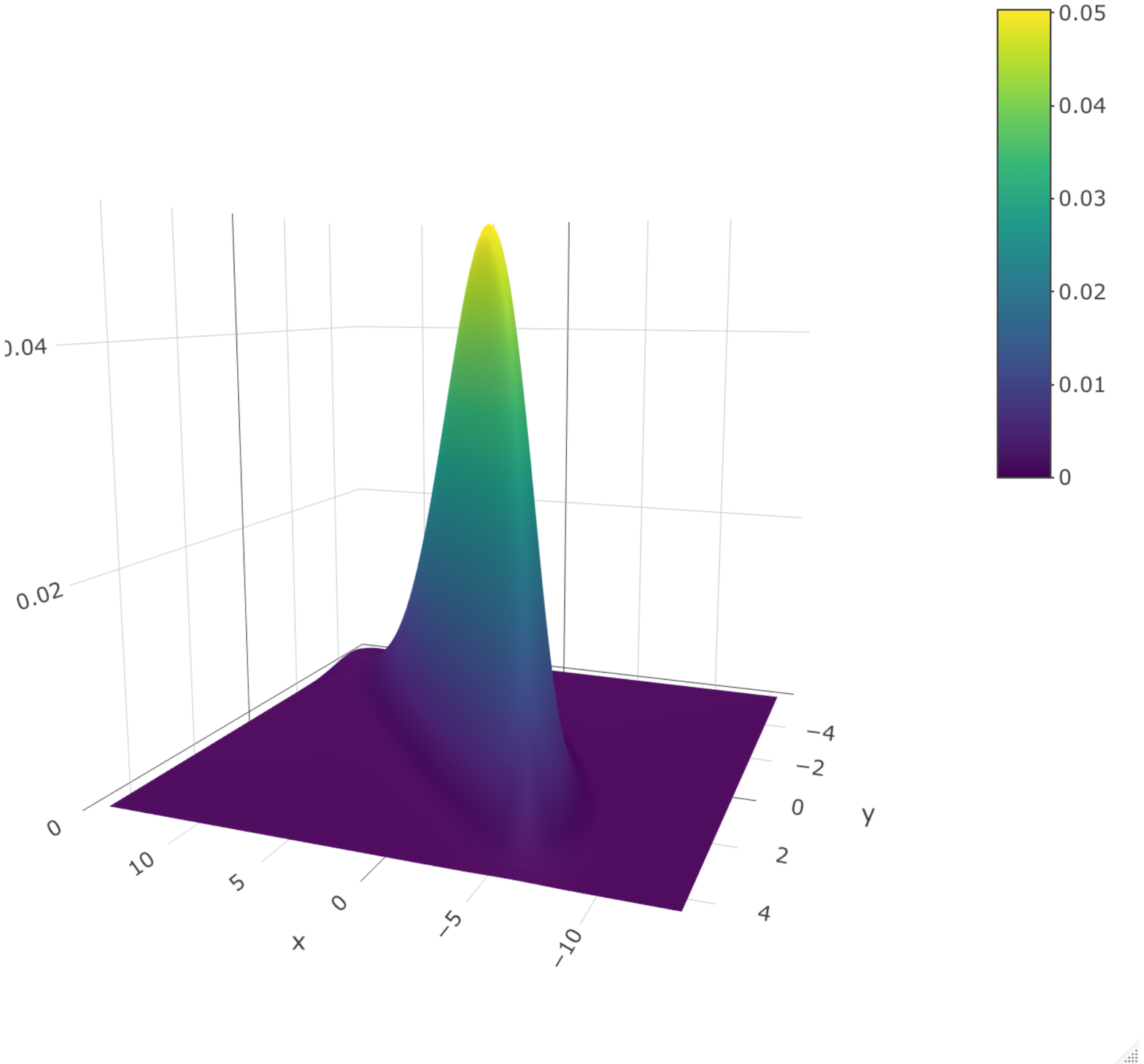}
\caption{ 3D plot of the reference solution computed by FE.}
\label{fig1}
\end{center}

\begin{center}
\includegraphics[angle=270,width=9cm]{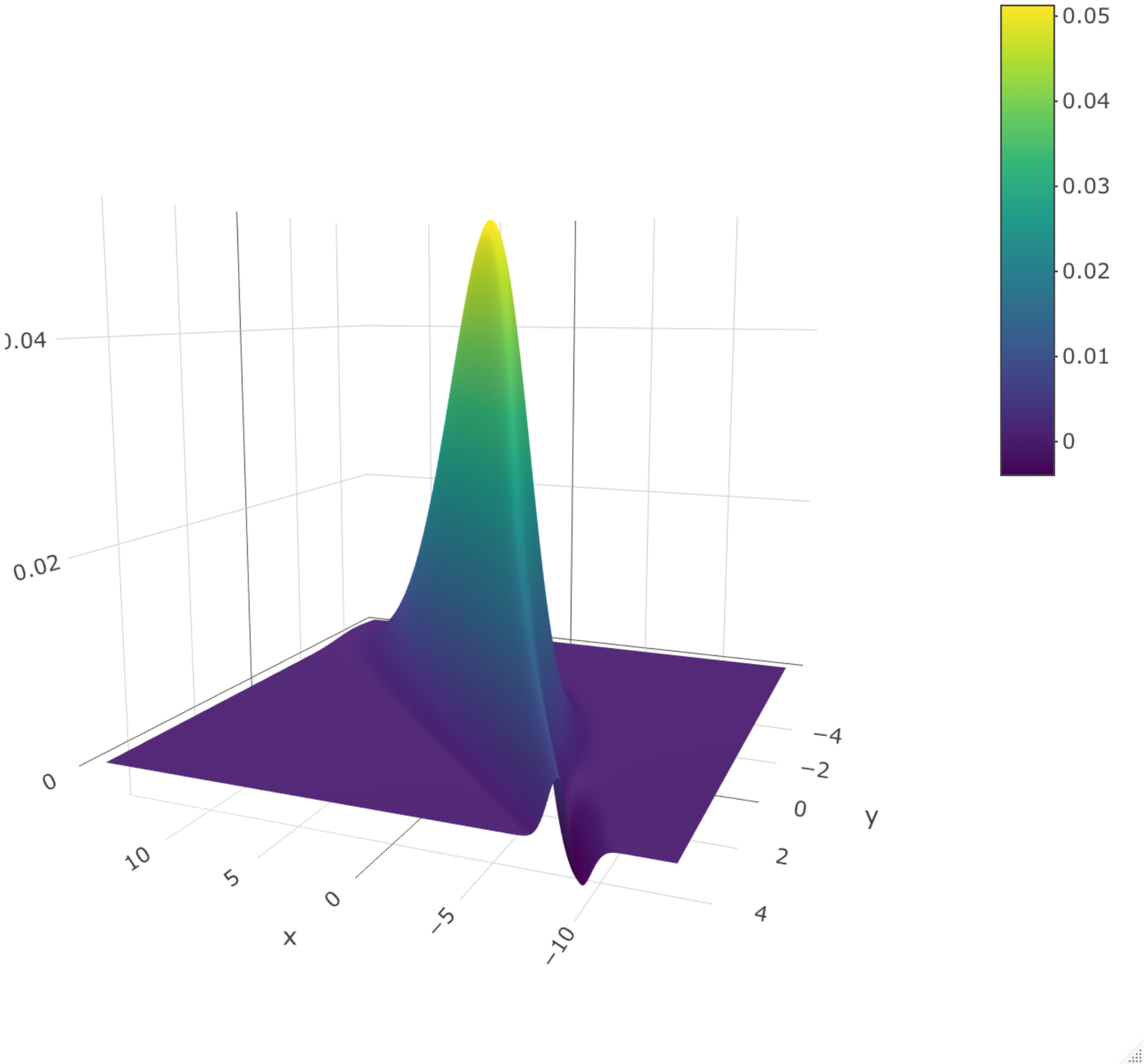}
\caption{ 3D plot of the solution computed by one iteration of the semigroup (i.e. \eqref{eq:approN} with $N=1$).}
\label{fig2}
\end{center}
\end{figure}
Note that the solution computed with $N=1$ shows some instability.
To illustrate this more precisely we plot on Figure \ref{fig3} approximated graphs of the function
$x\mapsto u(T=2.5,x,y=3.74)$ computed with the four methods (FE, MC with $n=10^4$ and $M=5\times 10^5$, and the iterated semigroup \eqref{eq:approN} with $N=1$ and $N=5$). We see that the solutions computed by FE, by MC and by \eqref{eq:approN} with $N=5$  are very close, while the one for $N=1$ is different and obviously presents issues.

\begin{figure}
\begin{center}
\includegraphics[width=9cm]{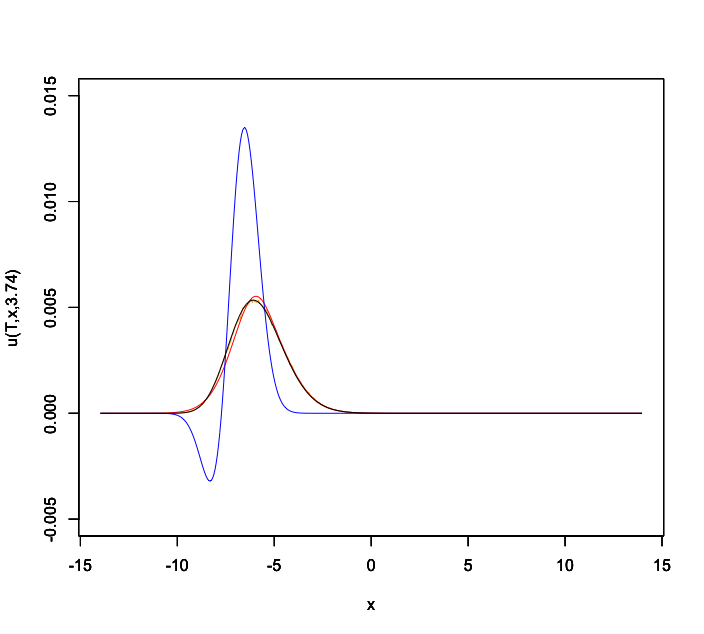}
\caption{ Plot of an approximation of the function $x\mapsto u(T=2.5,x,y=3.74)$, by FE (red), MC (orange), and the iterated semigroup \eqref{eq:approN} with $N=1$ (blue) and $N=5$ (black).}
\label{fig3}
\end{center}
\end{figure}

\section{Appendix}\label{rem:back-for}
In this appendix, we come back to the connection between the solution to a KHE in the backward or forward form and the probability transition function of processes governed by SDEs. As mentioned in the introduction, the transition probability function of the process described by \eqref{hypo1} is the solution to the Kolmogorov (hypoelliptic) equation in the backward form, that is \eqref{Kolm-back}. This is because the arrival point $(x',y')$ is fixed and we see
$p(t,x,y,x',y')$ as a function of time $t$ and starting point $(x,y)$.

If now we fix the starting point $(x,y)$, it is well known that $p(t,x,y,x',y')$ solves, as a function of the arrival point $(x',y')$ and time $t$, the KHE in the forward form
\begin{equation}\label{Kolm-for}
\left\{\begin{array}{rcl}
\partial_tp(t,x,y,x',y')&=&\frac12\Delta_{y'}p(t,x,y,x',y')-\sum_{i=1}^d\partial_{y_i'}[b_i(y') p(t,x,y,x',y')] \\
&&-<c(y'),\nabla_{x'}p(t,x,y,x',y')>,
 (t,x',y')\in\R_+^*\times\R^{\tilde d}\times\R^d\\
p(0,x,y,x',y')&=& \delta_{x}(x')\otimes\delta_{y}(y'),\quad (x',y')\in\R^{\tilde d}\times\R^d.
\end{array}\right.
\end{equation}
This is because the formal adjoint of the generator 
$$\cL=\frac12\Delta_y+<b(y),\nabla_y>+<c(y),\nabla_x>$$
of \eqref{hypo1} is defined by
$$
\cL^*f(x',y')=\frac12\Delta_{y'}f(x',y')-\sum_{i=1}^d\partial_{y_i'}[b(y')f(x',y')]-<c(y'),\nabla_{x'}f(x',y')>
$$
(see \cite[Chapter 5]{friedman-eds}; this is a consequence of Th. 5.4.7 which can be adapted to the hypoelliptic case). 
Another point of view is the following. Assume for simplicity that $b\equiv 0$. Equation~\eqref{Kolm-back} is the KHE in the backward form associated to the SDE described by
\eqref{hypo1}, but it also corresponds to the KHE in the forward form associated to
\begin{eqnarray}\label{hypo1-bis}
\begin{cases}
dX(t)=-c(Y(t))dt \\
dY(t)= dW(t).
\end{cases}
\end{eqnarray}
More precisely if for any $(x',y')$ one gets the solution $(t,x,y)\mapsto p(t,x,y,x',y')$ to \eqref{Kolm-back},
this defines a kernel $p(t,x,y,x',y')$ which is the transition function of \eqref{hypo1}; but by
setting $p^*(t,x',y',x,y):=p(t,x,y,x',y')$ one defines a new kernel which is the transition function of 
\eqref{hypo1-bis}. This is because the formal adjoint of the generator 
$\cL^*=\frac12\Delta_{y'}-<c(y'),\nabla_{x'}>$ of \eqref{hypo1-bis} is
$\cL=\frac12\Delta_{y}+<c(y),\nabla_{x}>$ (or by using directly \cite[Theorem~5.4.7]{friedman-eds}).

To sum up, solving \eqref{Kolm-back} with $b \equiv 0$ allows the computation of the transition function of \eqref{hypo1} (with $b\equiv 0$) or 
of \eqref{hypo1-bis}. But in these notes we will focus on the transition of \eqref{hypo1}, as this will be more coherent with our probabilistic computations.

\section{Acknowledgments}
This work was made with the support of the MathamSud FANTASTIC 20-MATH-05 project. The authors also thank the anonymous referees and the AE for valuable comments.


\begin{thebibliography}{plain}
\bibitem{Az:Doss} Azencott R., Doss H. L'équation de Schrödinger quand $\hbar$ tend vers zéro une approche probabiliste. 
Lect. Notes Math. 1109, 1-17 (1985).
\bibitem{Bian:Bian}Bian N.H., Emsile A.G., Kontar E.P.  A Fokker–Planck Framework for Studying the Diffusion of Radio Burst Waves in the Solar Corona. The Astrophysical Journal, Volume 873, Number 1 (2019).
\bibitem{Ca:Ca}Calin O., Chang D-C., Fan H. The Heat Kernel for Kolmogorov Type Operators and its Applications. J. Fourier Anal Appl  15: 816–838 (2009).
\bibitem{Ca1:Ca1}Calin O., Chang D-C.,Furutany K., Iwasaki C. Heat Kernel for Elliptic and Sub-elliptic Operators. Birkhäusser (2011).
\bibitem{Ca2:Ca2}Calin O.,  Chang D-C., Hu J.,  Li Y. Heat kernels for a class of degenerate elliptic operators using stochastic method. Complex Variables and Elliptic Equations
Vol. 57, Nos. 2-4, (2012).
\bibitem{DeW:DeW} Morette-DeWitt C. Feynman’s path integrals: definition without limiting procedure. Commun. Math. Phys. 28, 47–67 (1972).
\bibitem{fey:Hibs}Feynman, R. P.; Hibbs, A. R. Quantum mechanics and path integrals (International Series in Pure and Applied Physics). McGraw-Hill Publishing Company (1965).
\bibitem{friedman-eds}
A.~Friedman, \emph{Stochastic differential equations and applications}, Dover
  Books on Mathematics, Dover Publications, (2012).
\bibitem{Gzy:Leo} Gzyl H.,  León J.R. The Hamilton-Jacobi equation, the Feymann-Kac formula and the classical limit. Publicaciones Matemáticas del Uruguay
Volumen 17, 81-92 (2019).
\bibitem{Kac:Kac} Kac, M. Integration in function spaces and some of its applications. Lezioni Fermiane. Pisa: Accademia Nazionale dei Lincei, Scuola Normale Superiore (1980).
\bibitem{kara}
I.~Karatzas and S.E. Shreve, \emph{{Brownian motion and stochastic calculus.
  2nd ed.}}, {Graduate Texts in Mathematics, 113. New York etc.:
  Springer-Verlag. xxiii, 470 p. }, 1991.
\bibitem{Ko:Ko} Kolmogoroff, A. Zufällige Bewegungen. (Zur Theorie der Brownschen Bewegung.). Ann. Math. (2) 35, 116-117 (1934).
\bibitem{Re:Yo}  Revuz D, Yor M. Continuous Martingales and Brownian Motion. Springer-Verlag (1991).
\bibitem{Sch:Sch} Schulman, L. S. Techniques and applications of path integration. Reprint of the 1981 original.  Dover Publications, Inc. Mineola, New York (2005).
\bibitem{Si:Si} Simon B. Functional Integration and Quantum Physics. Academic Press, New York (1979).
\end{thebibliography}
\end{document}